\AddToHook{env/lemma/begin} {\crefalias{thm}{lemma}}
\AddToHook{env/prop/begin}{\crefalias{thm}{prop}}
\AddToHook{env/cor/begin}{\crefalias{thm}{cor}} 

\documentclass[11pt]{amsart}
\usepackage{cite,array}  
\usepackage{tikz,amscd}        
\usetikzlibrary{calc}  
\usetikzlibrary{arrows.meta, positioning, decorations.markings}
\usepackage[all]{xy}     
\usepackage{amssymb, amsmath,amsthm,color,amsfonts, mathrsfs}

\usepackage{hyperref}
\usepackage[capitalize,nameinlink,noabbrev,nosort]{cleveref}
\usepackage{subfiles}  
\usepackage{relsize}   
\usepackage{mathtools}  
\usepackage{pgfplots}   
\pgfplotsset{compat=1.18}

\definecolor{MilnorInk}{RGB}{30,47,65}
\definecolor{MilnorTeal}{RGB}{24,105,112}
\definecolor{MilnorBlue}{RGB}{57,81,142}
\definecolor{MilnorPlum}{RGB}{126,59,96}
\definecolor{MilnorOchre}{RGB}{144,101,38}
\tikzset{milnor path/.style={line width=.65pt,line cap=round,line join=round}}

\newtheorem{thm}{Theorem}[section]
\newtheorem{prop}[thm]{Proposition}
\newtheorem{lemma}[thm]{Lemma}

\newtheorem{cor}[thm]{Corollary}
\newtheorem{defn}[thm]{Definition}

\newtheorem{remark}[thm]{Remark}

\newtheorem{conjecture}[thm]{Conjecture}
\newtheorem{cond}[thm]{Condition}

 \numberwithin{equation}{section}

\newcommand{\dbb}[1]{[\![#1]\!]}

\newcommand{\bbC}{\mathbb C}
\newcommand{\bbP}{\mathbb P}

\newcommand{\bbZ}{\mathbb Z}
\newcommand{\bbR}{\mathbb R}

\newcommand{\cO}{\mathcal{O}}

\newcommand{\End}{\operatorname{End}}

\begin{document}

\title{Gamma conjectures for Milnor hypersurfaces}

\author{Jianxun Hu}
\address{School of Mathematics, Sun Yat-sen University, Guangzhou 510275, P.R. China}
\email{stsjxhu@mail.sysu.edu.cn}

\author{Hua-Zhong Ke}
\address{School of Mathematics, Sun Yat-sen University, Guangzhou 510275, P.R. China}

\email{kehuazh@mail.sysu.edu.cn}

\thanks{ 
 }
 
\author{Changzheng Li}
 \address{School of Mathematics, Sun Yat-sen University, Guangzhou 510275, P.R. China}
\email{lichangzh@mail.sysu.edu.cn}

\author{Jiayu Song}
 \address{ Hetao Institute of Mathematics and Interdisciplinary Sciences, Shenzhen 518045, P.R. China; School of Science and Engineering, The Chinese University of Hong Kong (Shenzhen), Shenzhen 518172, P.R. China }
\email{songjiayu@himis-sz.cn}

\thanks{}

\begin{abstract}
 We prove Gamma conjectures I and II of Galkin, Golyshev and Iritani for Milnor hypersurfaces, the smooth hypersurfaces of bidegree $(1,1)$ in the product of two projective spaces. 
 We  provide a  criterion for  Gamma conjecture II for general Fano manifolds, generalizing the criterion by Hu and Ke.
Using the mirror symmetry technique of Galkin and Iritani, we establish global Gamma-I for Milnor hypersurfaces, thereby proving  Gamma conjecture I as a special case. 
Using it, together with the Galois action and   our criterion, we verify Gamma conjecture II for Milnor hypersurfaces.   
 \end{abstract}

\maketitle

\tableofcontents

\section{Introduction}

The present article may be viewed as a companion of \cite{HKLS2}, in the sense that we approach Gamma conjecture II via the so-called ``global Gamma-I" as well. However, here we adopt a mirror-symmetric point of view instead,  and our proofs do not   rely on any results from \cite{HKLS2}. To illustrate the story, we also start with Dubrovin's conjecture proposed in his 1998 ICM address. This conjecture describes a deep relationship
between quantum cohomology and the derived category of coherent
sheaves \cite[Conjecture 4.2.2]{Dub}. For a Fano manifold $X$, the qualitative part of the
conjecture predicts that its big quantum cohomology is
semisimple if and only if its bounded derived category of coherent sheaves admits a full exceptional collection.
The quantitative part describes the Stokes matrix and the central connection matrix of the Dubrovin connection in terms of Euler pairings and
characteristic classes of exceptional objects. Gamma conjecture II, proposed by
Galkin, Golyshev, and Iritani \cite{GGI}, can be viewed as a refined version of the quantitative part. It relates the asymptotic behavior of the Dubrovin
connection, defined by genus-zero Gromov--Witten invariants,
to exceptional collections in the bounded derived category
of coherent sheaves. An independent refinement of Dubrovin's conjecture was proposed in \cite{CDG}.

Gamma conjecture II is expressed through Iritani's
$\widehat{\Gamma}$-integral structure \cite{Iri}, which defines
an integral lattice in the space of flat sections of the quantum
connection. Its $K$-group framing, denoted by $\mathcal Z^K$,
associates a flat section to a $K$-theory class. Assuming convergence and semisimplicity of big quantum cohomology,
and the existence of a full exceptional collection, Gamma
conjecture II asserts that, at a semisimple point and for a chosen
admissible phase, the flat sections associated with a suitable
full exceptional collection form an
\emph{asymptotically exponential fundamental solution} (AEFS)
of the Dubrovin connection \cite{GGI}.
As explained in Iritani's 2022 ICM report
\cite[Section~2]{Iri23icm}, it can be understood
as a compatibility statement between the Betti structure and
the Stokes structure discussed in \cite{HeSe,KKP}. The report also discussed Galkin's Gamma Conjecture III and the conjecture of Sanda and Shamoto \cite{SS20} as  generalizations of Dubrovin's conjecture to the non-semisimple setting.

Gamma conjecture II has so far been established for a limited
number of classes of Fano manifolds, with proofs drawing on
different geometric structures and analytic techniques.
Galkin, Golyshev, and Iritani \cite{GGI} proved the conjecture
for projective spaces and deduced it for Grassmannians using
the quantum Satake principle and isomonodromic deformation (see also \cite{CDG}).
For toric Fano manifolds, Fang and Zhou \cite{FaZh} combined
oscillatory integrals over Lefschetz thimbles with homological
mirror symmetry \cite{Kont} to identify the corresponding exceptional
objects.
Hu and Ke \cite{HuKe} proved the conjecture for smooth quadrics
by explicit asymptotic analysis of the flat sections associated
with Kapranov's exceptional collections, together with a
criterion based on isomonodromic deformation.
More recently, Hu, Ke, Li, and Su \cite{HKLS2} established
Gamma conjecture II for all del Pezzo surfaces through global
Gamma-I, combining Iritani's Galois action with operations on
exceptional collections and an analysis of the quantum spectrum.

In this paper, we study Gamma conjecture II for Milnor
hypersurfaces $H_{n,m}$, that is, smooth Fano hypersurfaces of bidegree $(1,1)$ in $\mathbb P^n\times\mathbb P^m$ with  $n\geq m\geq 1$.
When $m=1$, the Milnor hypersurface $H_{n,1}$ is a toric Fano manifold, so Gamma conjecture II follows from \cite{FaZh}. We therefore assume $m\geq2$ throughout this paper, except in Remark \ref{rmk:schu}.
Their two geometric descriptions make them a natural family
for studying the conjecture. Each $H_{n, m}$, as a hypersurface in $\mathbb P^n\times\mathbb P^m$,   admits a Landau-Ginzburg model from Givental's theory of toric complete intersections
\cite{Giv98}. Viewed as a Schubert subvariety of the two-step flag variety $F\ell_{1,n;n+1}$, it admits a Landau-Ginzburg model from toric degeneration similar to the construction in \cite{Giv97,BCFKS,HKLS}. The two mirrors are expected to be equivalent. At the same time, the projection to
$\mathbb P^m$ realizes $H_{n,m}$ as a $\mathbb P^{n-1}$-bundle,
and so Orlov's projective-bundle formula supplies a full
exceptional collection of line bundles \cite{Orl}.

Our main result is the following, proved in \textbf{\cref{thm:GammaIImainbody}}.

\begin{thm}\label{thm:introduction-gamma-II}
For every pair of integers $n\geq m\geq2$, Gamma conjecture II
holds for the Milnor hypersurface $H_{n,m}$.
\end{thm}

More precisely, we identify the flat sections associated with
a full exceptional collection from
Orlov's formula with an AEFS. Although both the mirror and the exceptional collection are
explicit, their existence does not by itself establish this
identification. The distinction is particularly clear from
the projective-bundle point of view. Iritani and Koto
\cite{IrKo} prove a formal decomposition of the quantum
$D$-module of a projective bundle into copies of the quantum
$D$-module of its base. As pointed out in \cite[Remark 1.11]{IrKo},
passing from this formal decomposition to an analytic
decomposition, and comparing the latter with Orlov's
semiorthogonal decomposition through the
$\widehat{\Gamma}$-integral structure, requires further
investigation. Gamma conjecture II
for the base and the fiber does not immediately yield the
conjecture for the total space from the formal decomposition alone.

Gamma conjecture I \cite{GGI} predicts, under the so-called Property
$\mathcal{O}$, that $\mathcal Z^K(\mathcal O_X)$ at $\tau=0\in H^2(X)$
has the smallest asymptotics along the positive
real axis.
It has been proved for several classes \cite{GGI, GZ16, GaIr, SS20, Ke24,HKLY, Yan22, Chow}.
Counterexamples to the original conjecture were found
in \cite{GHIKLS}, where modifications and an extension
over the K\"ahler moduli space were also investigated.
This inspires the authors of \cite{HKLS2} to study property Gamma-I, which is formulated at $\tau\in H^2(X)$ where the quantum spectrum has a simple rightmost eigenvalue.
The authors prove that, if this property holds at one
such point, it holds throughout the connected component
of the $(\mathrm{SR})$-region containing that point.
This parameter-dependent formulation provides the
starting point for our argument.
We obtain the following in  \textbf{\cref{Gamma-I}}.
\begin{thm}\label{mainthm:GammaI}
  Property Gamma-I
holds for $H_{n, m}$ at every $\tau\in H^2(H_{n,m};\mathbb R)$. Consequently, the original Gamma conjecture I holds for $H_{n, m}$. 
\end{thm}
  \noindent The proof uses Iritani's integral
representation for $\mathcal Z^K(\mathcal O)$ \cite{Iri11}, following the mirror-symmetric technique
of Galkin and Iritani \cite{GaIr}. The remaining problem is to pass from this single flat section
to the entire AEFS. Iritani's Galois action identifies
monodromy in the K\"ahler moduli with tensoring by line
bundles under the $K$-group framing \cite{Iri}. It therefore
provides the flat sections associated with the other objects
in Orlov's collection. However, knowing these sections separately
is not enough: they must have the required asymptotic behavior for a common admissible phase. 

To address this issue, we establish a criterion for Gamma
conjecture II  in  \textbf{\cref{criterion}}, generalizing \cite[Theorem 3.8]{HuKe} by Hu and Ke. Roughly speaking, the criterion
applies when the sections associated with a full exceptional
collection admit integral representations along pairwise
disjoint curved paths whose tails are parallel and ordered
compatibly with the collection. The proof uses Dubrovin's isomonodromic deformation theory for semisimple Frobenius manifolds \cite{Dub99}, which allows the starting points and the integration paths
to move simultaneously until the paths become parallel
half-lines and the corresponding sections form an AEFS. (See   \cref{figure_move} for an illustration.) Similar arguments were used by Iritani in \cite[Proof of Theorem 7.33]{Iri20} to establish Orlov-type decompositions of quantum $D$-modules in the toric setting. Our criterion is stated
for general Fano manifolds and is independent of the
particular geometry of Milnor hypersurfaces.

For the application to $H_{n,m}$, we study the quantum spectrum
near the boundary divisor $\{q_2=0\}$ of the K\"ahler moduli space.
Here $q_1,q_2$ are the small quantum parameters corresponding
to the hyperplane classes of the two ambient projective
spaces. Near the boundary divisor, the quantum spectrum forms  $n$ separated clusters, each containing
$m+1$ distinct eigenvalues (see   \cref{figure_eigen}). This distribution reflects
the projective-bundle structure of $H_{n,m}$.

The relation between the quantum spectrum and geometry
has been studied in several related settings.
Kontsevich conjectured that the quantum spectrum of a blow-up
decomposes into parts associated with the original variety
and the centre of the blow-up; see
\cite[Conjecture~1]{GS25}.
Gyenge and Szab\'o \cite{GS25} proved the surface case:
in a suitable limit of the quantum parameters, the spectrum
is asymptotic to that of a minimal model together with
additional eigenvalues tending to infinity, one for each
exceptional divisor.
At the level of quantum $D$-modules, Iritani
\cite{Iri23} and Iritani--Koto \cite{IrKo} established
formal decomposition theorems for blow-ups and projective
bundles, respectively. These theorems induce  formal decompositions of the
quantum spectrum and describe limiting configurations that reflect the corresponding geometric structures.

Applications to Gamma conjecture II require further
explicit control of the individual eigenvalues.
For del Pezzo surfaces, the argument in \cite{HKLS2}
requires finer eigenvalue estimates than those supplied
by the general asymptotic results of \cite{GS25}.
Similarly, in the present work, we need to determine
the separation of eigenvalues within each cluster.
Uniform asymptotic estimates provide the required separation and
allow us to construct the integration paths needed
in  \cref{criterion}.
Combining this analysis with Gamma-I and the Galois
action gives the desired AEFS.
These calculations further illustrate the viewpoint
developed in \cite[Sections~6 and~7]{GHIKLS}: the behavior
of the quantum spectrum near suitable boundary loci
reflects the geometry of extremal contractions.
Thus, the geometry  guides the choice of boundary regions,
while the applications to the Gamma conjectures call
for a more detailed analysis of the quantum spectrum
in those regions.

The paper is organized as follows. Section~2 recalls the
Dubrovin connection, the $\widehat{\Gamma}$-integral structure,
and Gamma conjecture II. Section~3 proves the criterion
based on Laplace integral representations. Section~4
describes the small quantum cohomology and the mirror of
Milnor hypersurfaces $H_{n, m}$. Section~5 proves property Gamma-I,
and Section~6 presents the proof of the Gamma conjecture II by using the Galois
action and the asymptotic distribution of the quantum
spectrum. Appendix~A gives the proof of the quantum ring
presentation, and Appendix~B establishes the uniform
asymptotic estimates used in Section~6.

\subsection*{Acknowledgements} 
The authors would like to thank Sergey Galkin, Hiroshi Iritani and Zhitong Su    for the relevant collaborations on Gamma conjectures that inspired the present  work.   
The first three authors are  supported in part by the National Key R \&  D Program of China No.~2023YFA1009801.
  J. Hu is also supported in part by NSFC Grant 12531003.
 H.-Z.  Ke is also supported in part by NSFC Grant 12271532.

\subsection*{AI Disclosure} After the initial draft had been completed, ChatGPT was used  to assist with proofreading,  minor mathematical checks and figure preparation.

\section{Preliminaries}
In this paper, we only consider the even part of  cohomology  for a Fano manifold $X$, and we set $H^{*}(X)=H^{{\rm ev}}(X;\mathbb{C}).$

 \subsection{Quantum cohomology} \label{quanh}
Let  $\overline{\mathcal{M}}_{0,k}(X,\beta)$ denote the   moduli  stack of  genus-zero, $k$-pointed  stable maps to $X$ of effective degree   $\beta\in {\rm Eff}(X)$. Here ${\rm Eff}(X)\subset H_2(X,\mathbb Z)$ is the semigroup generated by effective curve classes. For $1\leq i \leq k$, let ${\rm ev}_{i}:\overline{\mathcal{M}}_{0,k}(X,\beta)\longrightarrow X $  denote  the $i$-th evaluation map, and  $\mathcal{L}_i$ denote the universal cotangent line bundle on $\overline{\mathcal{M}}_{0,k}(X,\beta)$ 
with first Chern class $\psi_i$. 
 Given nonnegative integers $d_i$ and classes $\gamma_i\in H^{*}(X)$,  we have the genus-zero descendant Gromov-Witten invariant (see e.g. \cite{CoKa} for more details)
 $$\langle  \gamma_1 \psi_1^{d_1}, \gamma_2\psi_2^{d_2}, \cdots,  \gamma_k\psi_k^{d_k}  \rangle_{0,k, \beta}=\int_{[\overline{\mathcal{M}}_{0,k}(X,\beta)]^{\rm vir}} \prod_{i=1}^{k}(  {\rm ev}_{i}^{*}(\gamma_{i}) \cup \psi_i^{d_i}). $$
Here  $[\overline{\mathcal{M}}_{0,k}(X,\beta)]^{\rm vir}$ is the virtual fundamental class, which is of degree $\dim X+k-3+\int_{\beta}c_1(T_X)$ in  the Chow group $A_*(\overline{\mathcal{M}}_{0, k}(X, \beta))$.
We simply denote  $\gamma  \psi_i^{0}$ as $\gamma$ whenever there is no confusion.

 Choose a nef integral basis ${\xi_1,\cdots,\xi_r}$  of $H^{2}(X;\mathbb{Z})$/torsion  
and extend $\{\xi_0=\mathbf{1}, \xi_1, \xi_2,\cdots, \xi_r\}$ to a   homogeneous basis $\{\xi_i\}_{i=0}^{N-1}$  of $H^{*}(X)$. 
Let $\{\xi^j\}_j$ be the dual basis of $\{\xi_i\}_i$  with respect to the Poincar$\acute{\mbox{e}}$ pairing, namely   $(\xi_i,\xi^j)=\int_{X} \xi_{i} \cup \xi^{j} =\delta_{i,j}.$ 
Write $\tau=\sum_{i=0}^{N-1}t_{i}\xi_i \in H^{*}(X)$. The big quantum cohomology  $QH_{\rm big}(X)=(H^{*}(X)\otimes \bbC\dbb{ \tau}, *_{\tau})$  is a $\bbC\dbb{\mathbf{\tau}}$-algebra   with the big quantum product defined by  
$$\xi_{i} \ast_{\tau} \xi_{j}=\xi_i\cup\xi_j+ \sum_{l=0}^{N-1} \partial_{t_i} \partial_{t_j} \partial_{t_l} \Phi(\tau) \xi^{l}.$$ Here   
 the  potential function $\Phi(\tau)$ (with the specializations of  all   Novikov variables   at $1$ for simplicity) is given by   $$\Phi(\tau)= \sum_{\beta \in {\rm Eff}(X)\setminus\{0\}}\sum_{k=0}^{\infty}\frac{1}{k!} \langle \tau,  \cdots,  \tau \rangle_{0,k, \beta}\in\mathbb{C}\dbb{t_ 0,q_1,...,q_{r},t_{r+1},...,t_{N-1}}, $$
 where $q_i:=e^{t_i}$ for $1\leq i\leq r=\dim_\bbC H^2(X)$.

 \begin{defn}
   The big quantum cohomology  $QH_{\rm big}(X)$  {converges near the large radius limit} if there exists $\delta>0$ such that for $$|t_0|, |q_1|,\cdots,|q_r|, |t_{r+1}|,...,|t_{N-1}|<\delta,$$  $\Phi(\tau)$ absolutely converges as a series in $t_0,q_1,\cdots,q_{r},t_{r+1},\cdots,t_{N-1}$.
 \end{defn}  

 Let  $B \subset H^*(X)$ be the domain of convergence of   $\Phi(\tau)$ for $X$. 
We denote the quantum cohomology at $\tau \in B$ as $(H^{*}(X), \ast_{\tau})$. 
  \begin{defn}
       The big quantum cohomology $QH_{\rm big}(X)$  is called  \emph{semisimple} if there exists $\tau \in B$ such that $(H^*(X),\ast_{\tau})$  is isomorphic to a direct sum $\mathbb{C} \oplus \cdots \oplus \mathbb{C}$ as a $\bbC$-algebra. 
  \end{defn}

\subsection{Dubrovin connection} \label{Dub_connection}

 The  Dubrovin connection  $\nabla$   is a meromorphic flat connection on the trivial $H^*(X)$-bundle over $B\times\mathbb{P}^1$, given by 
\begin{align*}
    \nabla=\mathrm{d}+\sum_{i=0}^{N-1}{1\over z}(\xi_i*_\tau)\mathrm{d}t_i-\left({1\over z}\left(E*_{\tau}\right)-\mu\right){\mathrm{d}z\over z}.
\end{align*}
Here $z$ is the inhomogeneous coordinate on $\mathbb{P}^1$, 
  $E:=c_{1}(T_X)+\sum_{k=0}^{N-1}(1-\frac{1}{2} \deg \xi_k)t_k \xi_k$   is the Euler vector field,  
 and $\mu \in {\rm End}(H^{*}(X))$ is the  grading operator defined by $\mu(\xi_i)={1\over2}\left(\deg \xi_i-\dim_\bbC  X\right)\xi_i$. 
Along the $z$-direction, $\nabla$ has a regular singularity at $z=\infty$ and an irregular singularity at $z=0$. 
 
 Let $\widetilde{\bbC^\times}$ be the universal cover of $\bbC^\times=\bbP^1\setminus\{0,\infty\}$, and denote by  $\widetilde{\nabla}$   the pullback of $\nabla$ via $B\times\widetilde{\bbC^\times}\to B\times\bbC^\times$. Denote
$$\mathcal{S}:= \{\widetilde{\nabla}\mbox{-flat sections  of the trivial } H^*(X)\mbox{-bundle over } B\times\widetilde{\bbC^\times}\}.$$

\subsubsection{Flat sections around $z=\infty$}
    
 There is a $\bbC$-linear isomorphism  
\begin{align*}
  \mathcal{Z}: &\,\,\, H^*(X)\xrightarrow{\cong}\mathcal{S};\,\, \alpha \mapsto \left((\tau,z)\mapsto L(\tau,z)z^{-\mu}z^{c_1(X)\cup}\alpha\right).
\end{align*}
Here $z^{-\mu}z^{c_1(X)\cup}:=\mathrm{e}^{-\mu\log z}\mathrm{e}^{\left(c_1(X)\cup\right)\log z}$, and the formal endomorphism $L(\tau,z)\in\End(H^*(X))\dbb{{z^{-1}}}$ (following from \cite[Corollary 6.2]{Giv}) is given by 
\begin{align*}
    L(\tau,z)\alpha:=e^{-\frac{\eta}{z}}\alpha+\sum_{m, k \geq 0; \beta, l}  ({-1\over z})^{m+1}\frac{1}{k!}\langle  \psi_1^{m}({e^{-\frac{\eta}{z}} \alpha}), \xi_l,\tau', \cdots, \tau'  \rangle_{0,k+2, \beta} e^{\int_{\beta} \eta} \xi^{l},
\end{align*}
where  $\eta$ denotes the $H^2(X)$-component of $\tau$ and $\tau':=\tau-\eta$ and $(k,\beta) \neq (0,0)$.

  The Gamma class $\widehat{\Gamma}_{X}$    of $X$    is defined by $$ \widehat{\Gamma}_{X}:=\prod_{i=1}^{\dim X}\Gamma(1+\delta_i) \in H^*(X),$$ where $\delta_i$ are the Chern roots of the tangent bundle $T_X$ of $X$ and $\Gamma(x)$ denotes Euler's Gamma function.   Let $K(X)$ be the Grothendieck group of topological complex vector bundles on $X$,  and  ${\rm ch}$ denote the Chern character.

    \begin{defn} \label{Gamma structure}
     Consider the   group homomorphism 
     $$\mathcal{Z}^K: K(X) \xrightarrow{\quad}\mathcal{S};\,\,
E \mapsto (2\pi)^{-\frac{\dim X}{2}}\mathcal{Z}\Big(\widehat{\Gamma}_X\cup \sum\nolimits_{p\geq 0}(2\pi \mathbf{i})^p{\rm ch}_p(E)\Big).$$
The image $\mathcal{S}_{\mathbb{Z}}:=\mathcal{Z}^K(K(X))$ is called the  $\widehat{\Gamma}$-integral structure of $X$, which was introduced by Iritani \protect{\cite{Iri}}.   
\end{defn}

\subsubsection{Flat sections around $z=0$} \label{z=0}
We assume that the big quantum cohomology $QH_{\rm big}(X)$  is \emph{semisimple}. Then  $B\setminus B_{ss}$ is a divisor in $B$, where
   $$B_{ss}:=\{\tau \in B\mid (H^*(X), *_\tau) \mbox{ is semisimple}\}.$$

 As proved by Dubrovin \cite[Theorem 3.1, Lemma 3.2]{Dub99}, for a sufficiently small  domain $U\subset B_{ss}$,    there are holomorphic coordinates $\{u_i\}$ on $U$ (called \emph{canonical coordinates}), such that for $\tau\in U$, $(H^{*}(X),\ast_{\tau})$ has an idempotent basis $\{\pi_1(\tau),\cdots, \pi_N(\tau)\}$ (i.e. $\pi_i(\tau) \ast_{\tau} \pi_{j}(\tau)=\delta_{i,j}\pi_{i}(\tau)$), and $E(\tau)\ast_{\tau}\pi_{i}(\tau) =u_i(\tau)\cdot\pi_{i}(\tau)$. The functions $u_1,...,u_{N}$ are unique up to ordering. 
 We further assume $U$ to be simply-connected, so that the \emph{normalized idempotent} $\Psi_i(\tau):=\frac{\pi_{i}(\tau)}{\sqrt{(\pi_{i}(\tau),\pi_{i}(\tau))}}$   is a well-defined holomorphic map on $U$ (up to sign). For such a set-up, we simply say that $U$ is \emph{properly-chosen} with respect to  $\{(u_i,\Psi_i)\}_i$.

 We say that a phase $\theta \in \mathbb{R} $ is admissible at $\tau$ if $e^{\mathbf{i} \theta}$ is not parallel to any nonzero difference $u_i(\tau)-u_j(\tau)$ for all $1\leq i,j \leq N$.

\begin{prop}[\!\!\!\protect{\cite[Proposition 2.5.1]{GGI}}]\label{prop:AEFS}
      Suppose that $U \subset B_{ss}$ is properly-chosen with respect to $\{(u_i,\Psi_i)\}_i$, and $\theta \in\bbR$ is an admissible phase at some $\tau_0 \in U$. 
     Then there exist  a unique  analytic basis $\{y_i\}_{i=1}^N$ of $\mathcal{S}$,    $\epsilon>0$ and an open neighborhood $U''\subset U$ of $\tau_0$,    such that 
     \begin{align*} 
         y_{i}(\tau,z) e^{\frac{u_i(\tau)}{z}} \rightarrow \Psi_{i}(\tau), ~\text{as }  z \rightarrow 0 ~\text{ in the sector } |\arg z -\theta| < \frac{\pi}{2}+\epsilon,
     \end{align*}
 for any $\tau\in U''$ and any   $i$.  

\end{prop}

\begin{defn}
The basis $\{y_i\}_i$ of $\mathcal{S}$ in Proposition \ref{prop:AEFS}  is called the \emph{asymptotically exponential fundamental solution} (\emph{AEFS} for short) near $\tau_0$ respecting the normalized idempotents  $\{\Psi_i\}_i$ with phase $\theta$. 
\end{defn}

\subsection{Gamma conjecture II } \label{statement of GammaII}
 Roughly speaking, Gamma conjecture II of Galkin-Golyshev-Iritani 
   expects that an AEFS lies in the $\widehat{\Gamma}$-integral structure $\mathcal{S}_\bbZ$. 
   We first recall the relevant derived category terminology.
     \begin{defn}
    An object $F$ in a bounded derived category  $\mathcal{D}^{b}(X)$ of coherent sheaves on $X$ is called exceptional if $~{\rm Hom} (F,F)=\mathbb{C}$ and ${\rm Hom}(F,F[j])=0$ for all $j\neq 0$.  An ordered tuple of exceptional objects $E_1,\cdots, E_{m}$ is called an \textbf{exceptional collection} if ${\rm Hom} (E_i,E_{j}[k])=0$ for all $i,j,k$ with $i>j$ and $k\in \mathbb{Z}$. An exceptional collection is called \textbf{full} if it generates the whole category $\mathcal{D}^{b}(X)$.
\end{defn}

\begin{conjecture}[Gamma  conjecture II] \label{Gamma conjII}
Assume that: (1) big quantum cohomology of $X$ converges and is semisimple; (2)   $\mathcal{D}^{b}(X)$ admits a full exceptional collection.
 Suppose that $U\subset B_{ss}$ is properly-chosen with respect to $\{(u_i,\Psi_i)\}_i$,  and $\phi\in\bbR$ is an admissible phase at some $\tau_0 \in U$, where   $u_i$ are ordered so that $\mathrm{Im}\left(e^{-\mathbf{i}\phi}u_1(\tau_0)\right)\ge...\ge\mathrm{Im}\left(e^{-\mathbf{i}\phi}u_{N}(\tau_0)\right).$ 
     Let $\{y_i\}_i$ be the corresponding AEFS.       
     Then there exists a full exceptional collection $(E_1,\cdots,E_{N})$ of $\mathcal{D}^b(X)$ such that for each $1\leq i \leq N$,  
       $$y_i(\tau_0,z)=\mathcal{Z}^{K}(E_{i})(\tau_0,z).$$
\end{conjecture}

\begin{remark}
 By \cite[Remark 4.6.3]{GGI} and \cite[Remark 4.13]{GaIr}, 
      the validity of Gamma conjecture II does not depend on the choices of $\tau_0$ and $\phi$. 
\end{remark}

\section{A criterion for Gamma conjecture II}

Throughout this section, we assume that the big quantum cohomology $QH_{\rm big}(X)$  is \emph{semisimple}.

\subsection{AEFS via Laplace transformation }\label{sec:AEFSviaLaplace}

We follow \cite[Section 2]{GGI} to briefly review the construction of AEFS via Laplace transformation.

Let $U\subset B_{ss}$ be properly-chosen with respect to $\{(u_i,\Psi_i)\}$. Consider the Laplace-dual connection $\widehat{\nabla}$ of $ \nabla $. 
It is a meromorphic flat connection on the trivial $H^{*}(X)$-bundle
over $U \times \bbC_{\lambda} $, where $\bbC_{\lambda}$ is a copy of $\bbC$ with coordinate $\lambda$.
Let $D_i$ be the smooth divisor in $U \times \bbC_{\lambda} $ defined by $\lambda=u_i $, and let $D_{\infty}:=U \times \{\infty\}$
be the divisor at infinity. Then $\widehat{\nabla}$ has only logarithmic singularities at $D_1, \cdots,D_N, D_{\infty}$. For $(\tau_0,\lambda_0)\in D_i$, there exists a unique $\widehat{\nabla}$-flat section $\hat{y}_{i}(\tau,\lambda)$ which is holomorphic near $(\tau_0,\lambda_0)$ and $\hat{y}_{i}(\tau, u_i(\tau))=\Psi_i(\tau)$. Let $\phi$ be an admissible  phase at $\tau_0$, and we define the Laplace transform of $\hat{y}_i$ associated to $\phi$ with respect to $u_i(\tau)$ by

\begin{align}
    y_i(\tau,z)=\frac{1}{z}\int_{u_i(\tau)+\bbR_{\geq 0} e^{\mathbf{i}\phi} } \hat{y}_i(\tau,\lambda) e^{-\frac{\lambda}{z}}d \lambda.  
\end{align}
Then $\{y_i\}$ is the AEFS respecting $\{\Psi_i\}$ with phase $\phi$.

\subsection{Analytic continuation of Dubrovin connection} \label{B_ts}
 Define 
$$B_{ts}:=\{\tau \in B_{ss} | \mbox{ eigenvalues of} ~E\ast_{\tau} \mbox{are pairwise distinct } \} $$
Let $U^{'} \subset B_{ts}$ be a properly-chosen neighborhood of $\tau_0 \in B_{ts}$ with respect to $\{(u_i, \Psi_i)\}_{1 \leq i \leq N}$. Recall that both $t_i$'s and $u_i$'s are coordinates on $U'$, so  $t_i$'s can be regarded as functions of the chosen canonical coordinates $u_i$'s.
Set $$\Delta := \{(u_1, \ldots , u_N) \in \mathbb{C}^N: u_i = u_j \text{ for some } i \neq j\},$$ 
and write $\mathbf{w}_0 := (u_1(\tau_0), u_2(\tau_0),\ldots , u_N(\tau_0)) \in \mathbb{C}^N \backslash \Delta .$ 
Let $p:\mathcal{C} \rightarrow \mathbb{C}^N \backslash \Delta $ be the universal covering based at $\mathbf{w}_0 $,  realized as homotopy classes of paths starting at $\mathbf{w}_0$. By abuse of notation, we still denote by $\mathbf{w}_0$ the point in $\mathcal{C}$ corresponding to the trivial loop at $\mathbf{w}_0$, and use the chosen canonical coordinates $u_i$'s to identify the open neighborhood $U^{'}$ of $\tau_0$ with a neighborhood of $\mathbf{w}_0$ in $\mathcal{C}$. Then via the diagram
\[
\begin{CD} 
U'  \xhookrightarrow{\qquad}   \mathbb{C}^N\backslash\Delta @<{~p}<<  \mathcal{C}, \\
\end{CD}
\]
we regard $u_i$'s as holomorphic functions on  $ \mathcal{C}$.
From \cite[Theorem 4.5, Theorem 4.6]{Dub99}, the Dubrovin connection  admits a unique isomonodromic deformation over $\mathcal{C}$. We still denote it by $\nabla$ by abuse of notation. 
More precisely, there is a unique meromorphic flat connection $\nabla$ on the trivial $H^* (X)$-bundle over $\mathcal{C} \times \mathbb{P}^1$, which reads with respect to a frame of constant sections
    \begin{align*}
    \nabla=\mathrm{d}+\sum_{j=1}^{N}{\mathscr{C}_j\over z} \mathrm{d}u_j-\left({\mathscr{U} \over z}-\mu\right){\mathrm{d}z\over z}.
\end{align*}
Here $\mathscr{C}_j$ and $\mathscr{U}$ are $\mathrm{End}(H^*(X))$-valued meromorphic functions on $\mathcal C$, such that $\nabla$ restricts to the Dubrovin connection on $U'\times\mathbb P^1$, and $ \mathscr{C}_j (\Psi_i)=\delta_{ij}\Psi_i$, $\mathscr{U}(\Psi_i)=u_i \Psi_i$.

This isomonodromic deformation defines a semisimple Frobenius manifold structure on a dense open subset $\mathcal{B}$ of $\mathcal{C}$, such that the complement $\mathcal{C} \backslash \mathcal{B}$ is a divisor in $\mathcal C$.  We have $U^{'} \hookrightarrow \mathcal{B}$ via the above mentioned identification, and
$\mathcal{B}$ is connected.
Now we can analytically continue the functions $u_i$ and $\Psi_i$ from $U'$ to $\mathcal{B}$. Moreover, as functions of the chosen canonical coordinates on $\mathcal{B}$, from \cite[(4.54)]{Dub99}, the functions $t_i$'s can be analytically continued to $\mathcal{B}$. So we  use this analytic continuation to define cohomology framing and $K$-group framing on $\mathcal{B}$, denoted as $\mathcal{Z}$ and $\mathcal{Z}^{K}$ by abuse of notation.

Moreover, the Laplace-dual connection $\widehat{\nabla}$ can also be analytically continued to $\mathcal{B} \times \mathbb{C}_\lambda$ (see also \cite[Lecture 5]{Dub99}).
 Laplace transformation of $\widehat{\nabla}$-flat sections gives a local AEFS for ${\nabla}$.
 We can therefore formulate Gamma conjecture II for $X$ over $\mathcal{B}$, that is, for $(\mathbf{w}, \phi) \in \mathcal{B} \times \mathbb{R}$ with  admissible phase $\phi$ at $\mathbf{w}$, matching the corresponding AEFS with flat sections from a full exceptional collection via $\mathcal{Z}^{K}$. Now it follows from \cite[Section 4]{GGI} that the validity of Gamma II does not depend on the choice of $(\mathbf{w}, \phi)$.

\subsection{A criterion for Gamma conjecture II}

 For a smooth embedding $$\gamma: \mathbb R_{\ge0}\rightarrow \mathbb C_\lambda,$$ we say that it is \textbf{admissible at $\tau_0$ with end-phase $\phi\in\mathbb R$}, if $\gamma(0)\in{\rm Spec}(E*_{\tau_0})$, $\gamma(\mathbb R_{>0})\cap{\rm Spec}(E*_{\tau_0})=\varnothing$, and $\gamma(t)$ is linear with $\arg\gamma'(t)=\phi$ near $t=\infty$. Here $\mbox{Spec}(E*_{\tau_0})$  denotes  the multiset of eigenvalues of the linear operator $E*_{\tau_0}$ on $(H^*(H_{n, m}),*_{\tau_0})$.

The following theorem gives a criterion for Gamma conjecture II.

\begin{thm} \label{criterion}
For a Fano manifold $X$, let $\mathcal{E} = (E_{1}, \ldots, E_{N})$ be a full exceptional collection of $\mathcal{D}^{b}(X)$. Let $U^{'} \subset B_{ts}$ be a properly-chosen neighborhood of $\tau_0 \in B_{ts}$ with respect to $\{(u_i, \Psi_i)\}_{1 \leq i \leq N}$.
Assume that there exist $t_0\in\mathbb R_{\geq 0}$, $\phi_1,\dots,\phi_N\in\mathbb R$ with $\phi_1\ge\dots\ge\phi_N$ and $\phi_1-\phi_N<2\pi$,  and admissible paths $\gamma_i$ at $\tau_0$ with end-phase $\phi_i$,  such that
\begin{enumerate}
\item[(i)] $\gamma_i(0)=u_i(\tau_0)$, $\gamma'_i(t)$ is constant for $ t \ge t_0$;
 \item[(ii)] $\mathcal{Z}^{K}(E_{i})(\tau_0,z)=\frac{1}{z}\int_{\gamma_i } \hat{y}_i(\tau_0,\lambda) e^{-\frac{\lambda}{z}}d \lambda$ when $\mathrm{Re}(\frac{e^{\mathbf i\phi_i}}{z})>0$, where $\hat y_i$ is as in Section \ref{sec:AEFSviaLaplace};
\item[(iii)] if $\phi_i=\phi_j$ for $i<j$, then ${\rm Im} (e^{-\mathbf{i}\phi_i} \gamma_{i}(t)) > {\rm Im} (e^{-\mathbf{i}\phi_j} \gamma_{j}(t))$   when  $t\ge t_0$;
\item[(iv)]  the images of $\{\gamma_i\}_{1\leq i \leq N}$ are pairwise disjoint.
\end{enumerate}
Then there exist $\mathbf{w}^{*} \in \mathcal{B}$ and $\phi^*\in\mathbb R$ such that:
\begin{enumerate}
\item[(1)] ${\rm Im} (e^{-\mathbf{i}\phi^*} u_{1}(\mathbf{w}^{*})) > {\rm Im} (e^{-\mathbf{i}\phi^*} u_{2}(\mathbf{w}^{*})) > \dots > {\rm Im} (e^{-\mathbf{i}\phi^*} u_{N}(\mathbf{w}^{*}))$;
\item[(2)] $\{\mathcal{Z}^{K}(E_{i})(\mathbf{w}^{*},z)\}_{i}$
is the AEFS near $\mathbf{w}^{*}$ respecting $\{\Psi_i\}_i$ with phase $\phi^*$.
\end{enumerate}
In particular, Gamma conjecture II holds for $X$.
\end{thm}

\begin{remark}
    This theorem is a generalization of \cite[Theorem 3.8]{HuKe}; here the admissible paths are allowed to be curved, and their end-phases are not required to be pairwise distinct.
\end{remark}

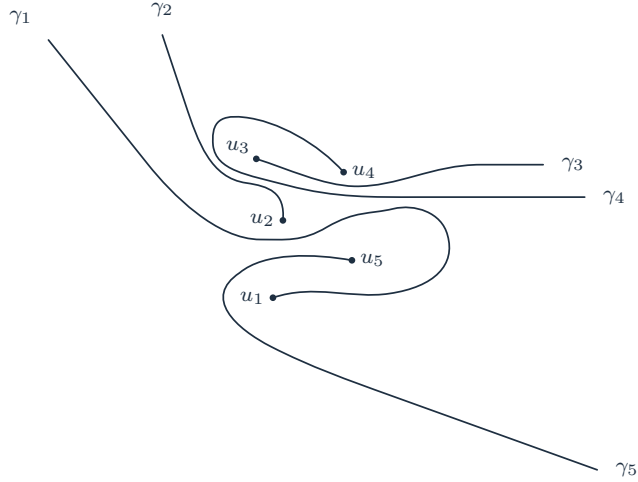
\begin{figure}[htbp]
\centering
\begin{tikzpicture}[
  x=1.1cm,y=1.1cm,
  every node/.style={font=\fontsize{8}{9.5}\selectfont},
  every path/.style={MilnorInk},
  line cap=round,
  line join=round
]

\draw[milnor path]
  (0.20,-0.65)
  .. controls (0.70,-0.48) and (1.15,-0.67) .. (1.60,-0.60)
  .. controls (2.15,-0.52) and (2.40,-0.25) .. (2.30,0.10)
  .. controls (2.23,0.38) and (1.90,0.48) .. (1.65,0.42)
  .. controls (1.35,0.35) and (1.20,0.40) .. (0.85,0.20)
  .. controls (0.55,0.02) and (0.35,0.05) .. (0.05,0.05)
  .. controls (-0.48,0.05) and (-1.02,0.60) .. (-1.30,0.95)
  -- (-2.50,2.45)
  node[above left=3pt] {$\gamma_1$};

\draw[milnor path]
  (0.32,0.28)
  .. controls (0.35,0.62) and (0.15,0.69) .. (-0.12,0.73)
  .. controls (-0.55,0.79) and (-0.70,1.22) .. (-0.85,1.67)
  -- (-1.13,2.51)
  node[above=3pt] {$\gamma_2$};

\draw[milnor path]
  (0.00,1.02)
  .. controls (0.38,0.90) and (0.80,0.70) .. (1.18,0.69)
  .. controls (1.70,0.68) and (2.12,0.95) .. (2.70,0.95)
  -- (3.45,0.95)
  node[right=3pt] {$\gamma_3$};

\draw[milnor path]
  (1.05,0.86)
  .. controls (0.75,1.20) and (0.25,1.48) .. (-0.15,1.52)
  .. controls (-0.48,1.56) and (-0.54,1.42) .. (-0.52,1.18)
  .. controls (-0.48,0.88) and (-0.12,0.83) .. (0.38,0.70)
  .. controls (0.95,0.55) and (1.40,0.56) .. (1.95,0.56)
  -- (3.95,0.56)
  node[right=3pt] {$\gamma_4$};

\draw[milnor path]
  (1.15,-0.20)
  .. controls (0.70,-0.12) and (0.10,-0.10) .. (-0.20,-0.35)
  .. controls (-0.65,-0.65) and (-0.25,-1.02) .. (0.18,-1.24)
  .. controls (0.65,-1.48) and (1.10,-1.64) .. (1.60,-1.82)
  -- (4.10,-2.72)
  node[right=3pt] {$\gamma_5$};

\foreach \x/\y in {
  0.20/-0.65,
  0.32/0.28,
  0.00/1.02,
  1.05/0.86,
  1.15/-0.20
}{
  \fill[MilnorInk] (\x,\y) circle (1.25pt);
}

\node[left=3pt,inner sep=0pt]       at (0.20,-0.65) {$u_1$};
\node[left=3pt,inner sep=0pt]       at (0.32,0.28)  {$u_2$};
\node[above left=3pt,inner sep=0pt] at (0.00,1.02)  {$u_3$};
\node[right=3pt,inner sep=0pt]      at (1.05,0.86)  {$u_4$};
\node[right=3pt,inner sep=0pt]      at (1.15,-0.20) {$u_5$};

\end{tikzpicture}
\caption{Contours $\gamma_i$ from
  $u_i(\widetilde b(0))$ in the $\mathbb{C}_\lambda$ plane.}
\label{fig:criterion-contours}
\end{figure}

\begin{proof}
We first sketch our proof. We will first find a path $\widetilde b(s):[0,1]\to\mathcal{B}$ starting from $\widetilde b(0)=\tau_0\in U'$ and some large $T$, such that  
\begin{itemize}
 \item[(i)] $u_i(\widetilde b(s))$ moves almost along $\gamma_i$ in $\mathbb{C}_\lambda$, and $u_i(\widetilde b(1))$ is near $\gamma_i(T)$;
\item[(ii)] the half lines  $L(u_{i}(\widetilde b(1)), \phi_i):=u_{i}(\widetilde b(1))+\bbR_{\geq 0} e^{\mathbf{i}\phi_i}$ are disjoint in $\mathbb{C}_{\lambda}$.
\end{itemize}
Then for each $i$, we analytically continue the $\widehat\nabla$-flat section $\widehat{y}_i$ along $u_i(\widetilde b(s))$; the Laplace transform of $\widehat{y}_i$ along suitably chosen contours (see Figures \ref{fig:criterion-contours} and \ref{fig:criterion-rays}) gives a family of $\nabla_{\partial_z}$-flat sections depending on $s$, and we show that they are all equal to $\mathcal{Z}^{K}(E_{i})$. Finally, we use the criterion of \cite[Theorem 3.8]{HuKe} at $s=1$ to conclude the theorem.

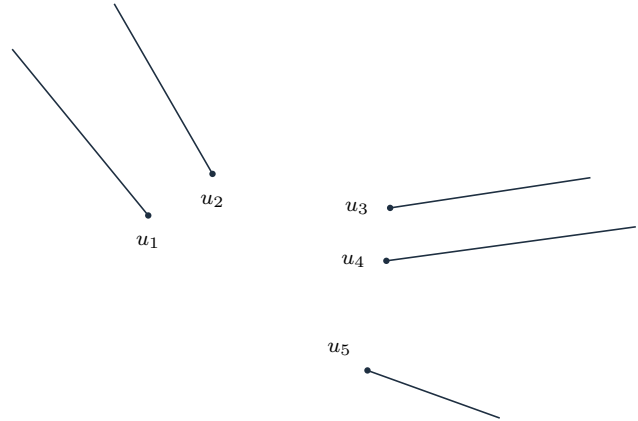
\begin{figure}[htbp]
\centering
\begin{tikzpicture}[
  x=1cm,y=1cm,
  every node/.style={font=\fontsize{8}{9.5}\selectfont},
  contour/.style={draw=MilnorInk,line width=.6pt}
]
\coordinate (u1) at (1.80,3.30);
\coordinate (u2) at (2.65,3.85);
\coordinate (u3) at (5.00,3.40);
\coordinate (u4) at (4.95,2.70);
\coordinate (u5) at (4.70,1.25);

\draw[contour] (u1) -- (0.00,5.50);
\draw[contour] (u2) -- (1.35,6.10);
\draw[contour] (u3) -- (7.65,3.80);
\draw[contour] (u4) -- (8.25,3.15);
\draw[contour] (u5) -- (6.45,0.62);

\foreach \i in {1,...,5}{
  \fill[MilnorInk] (u\i) circle (1.2pt);
}

\node[below=4pt]      at (u1) {$u_1$};
\node[below=4pt]      at (u2) {$u_2$};
\node[left=4pt]       at (u3) {$u_3$};
\node[left=4pt]       at (u4) {$u_4$};
\node[above left=3pt] at (u5) {$u_5$};
\end{tikzpicture}
\caption{Contours
$L(u_i(\widetilde b(1)),\phi_i)$ from
$u_i(\widetilde b(1))$ in the $\mathbb{C}_\lambda$ plane.}
\label{fig:criterion-rays}
\end{figure}

\medskip

\noindent  \textbf{The detailed arguments are presented in the following.}

Let $T\in\mathbb R_{>0}$ be such that $\gamma_i'(t)$ is constant for $t\ge T$. Set
$$
a(s):=
\bigl(\gamma_1(sT),\ldots,\gamma_N(sT)\bigr)\in\mathbb C^N,
\qquad 0\leq s\leq1.
$$ 
Since $\gamma_1,\dots,\gamma_N$ are pairwise disjoint, it follows that   $a(s) \in \mathbb{C}^N\setminus\Delta$. Now the path $a:[0,1] \rightarrow \mathbb C^N\setminus\Delta$  has a unique lift $\widetilde a:[0,1]\longrightarrow\mathcal C$, i.e. $p \circ \widetilde a=a$, satisfying $\widetilde a(0)=\tau_0$.
Here we regard $\tau_0$ as the point in $\mathcal C$ corresponding to the constant loop, and view $U'$ as an open neighborhood of $\tau_0\in\mathcal{C}$.

Note that
 there exists a positive number $d_0$ such that
\[
\operatorname{dist}
\bigl(\gamma_i([0,\infty)),\gamma_j([0,\infty))\bigr)> 6d_0,\textrm{ when }i\ne j.
\]
For $s\in[0,1]$, let $N_s$ be the open neighborhood of $\tilde a(s)$ in $\mathcal C$ such that $N_s\xrightarrow{p}p(N_s)\subset\mathbb C^N\setminus\Delta$ is a homeomorphism and $p(N_s)=B(u_1(\widetilde a(s)),d_0)\times\dots\times B(u_N(\widetilde a(s)),d_0)$. Here $B(u,d)$ is the open ball centered at $u$ with radius $d$. 
By compactness of $[0,1]$, we can choose \[0=s_0<s'_1<s_1<\dots<s'_k<s_k=1,\] such that $\widetilde a([s_0,s'_1])\subset N_{s_0}$, $\widetilde a([s'_k,s_k])\subset N_{s_k}$, and
\[
\widetilde a([s'_j,s'_{j+1}])\subset N_{s_j}\textrm{ for }j=1,\dots,k-1.
\]
Since $\mathcal C\setminus\mathcal B$ is a divisor in $\mathcal C$, it follows that $N_{s_j}\cap\mathcal B$ is non-empty and connected, and $N_{s_{j-1}}\cap N_{s_j}\cap\mathcal B$ is non-empty. Let $\tau_{s'_j}\in N_{s_{j-1}}\cap N_{s_j}\cap\mathcal B$. Then there is a path $\widetilde b:[0,1]\to\mathcal B$ satisfying: (1) $\widetilde b({s'_j})=\tau_{s'_j}$ for $j=1,\dots,k$ and $\widetilde b(0)=\tau_0$; (2) $\widetilde b([s_0,s'_1])\subset N_{s_0}$, $\widetilde b([s'_k,s_k])\subset N_{s_k}$, and
\begin{align*}
\widetilde b([s'_j,s'_{j+1}])\subset N_{s_j}\textrm{ for }j=1,\dots,k-1.
\end{align*}
Then one can check that
\[
\left|u_i(\widetilde b(s))-\gamma_i(sT)\right|<2d_0,\textrm{ when }0\leq s\leq1,\ 1\leq i\leq N.
\]
For $s\in[0,1]$, set $\delta_i(s):=u_i(\widetilde b(s))-\gamma_i(sT)$,
and define the path $\gamma_{s,i}:[0,+\infty)\to\mathbb C_\lambda$ by
\[
\gamma_{s,i}(t):=
\gamma_i(sT+t)+\delta_i(s),
\qquad t\geq0.
\]
Note that
$ 
\gamma_{0,i}=\gamma_i$, and we view $\gamma_{s,i}$ as a deformation of $\gamma_i$ with starting point $\gamma_{s,i}(0)=u_i(\widetilde b(s))$.
For $i\neq j$ and $t,t'\geq0$, we have
\[
\begin{aligned}
\left|\gamma_{s,i}(t)-\gamma_{s,j}(t')\right|
&\geq
\left|\gamma_i(sT+t)-\gamma_j(sT+t')\right|
-|\delta_i(s)|-|\delta_j(s)|\\
&>6d_0-2d_0-2d_0>0.
\end{aligned}
\]
Thus the paths $\gamma_{s,1},\dots,\gamma_{s,N}$ are pairwise disjoint, and consequently each $\gamma_{s,i}$ does not pass through $u_j(\tilde b(s))$ when $j\ne i$. Moreover, each $\gamma_{s,i}$ is an embedding. So we conclude that the path $\gamma_{s,i}$ is admissible at $\widetilde b(s)$ with end-phase $\phi_i$.

Set $F=\widetilde b \times {\rm id}:[0,1] \times \mathbb{P}^1 \rightarrow \mathcal{B}\times \mathbb{P}^1$. Then with respect to a frame of constant sections, the connection $F^*\nabla$ reads
       \begin{align*}
    F^*\nabla=\mathrm{d}+{1\over z}\sum_{j=1}^{N}\partial_s[u_j (\widetilde b(s))]\mathscr{C}_j(\widetilde b(s))\,\mathrm{d} s-\left({\mathscr{U}(\widetilde b(s)) \over z}-\mu\right){\mathrm{d}z\over z}.
\end{align*} 
Note that we can analytically continue $\widehat y_i(\tau,\lambda)$ from $(\tau,\lambda)=(\tau_0,u_i(\tau_0))$ to $(\tau,\lambda)=(\widetilde b(s),\gamma_{s,i}(t))$ as follows: first along the curve $s\mapsto(\widetilde b(s),\gamma_{s,i}(0))$, and then along the curve $t\mapsto(\widetilde b(s),\gamma_{s,i}(t))$. In particular, this obtained $\widehat \nabla$-flat section $\widehat y_i(\tau,\lambda)$ is holomorphic near $(\widetilde b(s),\gamma_{s,i}(0))$ and satisfies $\widehat{y_i}(\widetilde b(s),\gamma_{s,i}(0))=\Psi_i(\widetilde b(s))$. 
For $z$ satisfying $\operatorname{Re}(\frac{e^{\mathbf{i}\phi_i}}{z})>0$, define
\[
I_i(s,z):=
\frac1z\int_{\gamma_{s,i}}
\widehat y_i(\widetilde b(s),\lambda)
e^{-\frac{\lambda}{z}}\,{\mathrm d}\lambda=\frac1z\int_0^\infty\widehat{y}_i(\widetilde b(s),\gamma_{s,i}(t))e^{-\frac{\gamma_{s,i}(t)}{z}}\partial_t\gamma_{s,i}(t)\,\mathrm{d} t.
\]
This integral converges since $\lambda=\infty$ is a regular singularity of $\widehat\nabla$. We would like to show that
\[
I_i(s,z)=\mathcal Z^K(E_i)(\widetilde b(s),z).
\]
The equality holds for $s=0$, and therefore to this end, it suffices to verify that $ \left(F^{*}\nabla \right)I_i=0$.

Since  $\widehat y_i$ is $\widehat\nabla$-flat, it follows that
\begin{align*}
     \Big(\partial_s - \sum_{j=1}^{N}\partial_s[u_j (\widetilde b(s))]\mathscr{C}_j(\widetilde b(s))(\lambda-\mathscr{U}(\widetilde b(s)))^{-1} \mu \Big) \widehat y_i(\widetilde b(s),\lambda)=0,  \\
     \left(\partial_\lambda+(\lambda-\mathscr{U}(\widetilde b(s)))^{-1}\mu \right) \widehat y_i(\widetilde b(s),\lambda)=0,
\end{align*}
and consequently,
\begin{align}\label{eq--slambda}
    \partial_s[\widehat y_i(\widetilde b(s),\lambda)]=-\sum_{j=1}^N\partial_s[u_j (\widetilde b(s))]\mathscr C_j(\widetilde b(s))\partial_\lambda[\widehat y_i(\widetilde b(s),\lambda)].
\end{align}
For $(F^*\nabla)_{\partial_s}I_{i}(s,z)=0$, note that $\partial_sI_i(s,z)$ is equal to $(i)+(ii)+(iii)+(iv)$, where
\begin{align*}
    (i)&=\frac{1}{z} \int_{\gamma_{s,i}}\partial_s[\widehat y_i(\widetilde b(s),\lambda)]
e^{-\frac{\lambda}{z}}{\mathrm d}\lambda,\\
(ii)&=\frac1z\int_0^\infty(\partial_\lambda\widehat{y}_i)(\widetilde b(s),\gamma_{s,i}(t))\cdot\partial_s\gamma_{s,i}(t)e^{-\frac{\gamma_{s,i}(t)}{z}}\partial_t\gamma_{s,i}(t)\,\mathrm{d} t,\\
(iii)&=\frac1z\int_0^\infty\widehat{y}_i(\widetilde b(s),\gamma_{s,i}(t))e^{-\frac{\gamma_{s,i}(t)}{z}}\left(-\frac{\partial_s\gamma_{s,i}(t)}{z}\right)\partial_t\gamma_{s,i}(t)\,\mathrm{d} t,\\
(iv)&=\frac1z\int_0^\infty\widehat{y}_i(\widetilde b(s),\gamma_{s,i}(t))e^{-\frac{\gamma_{s,i}(t)}{z}}\partial_s\partial_t\gamma_{s,i}(t)\,\mathrm{d} t,
\end{align*}
and we set $(v)=\frac1z\sum_{j=1}^N\partial_s(u_j(\widetilde b(s)))\mathscr C_j(\widetilde b(s))I_{i}(s,z)$. From \eqref{eq--slambda}, we get
\[
(i)=- \frac{1}{z}\sum_{j=1}^{N}\partial_s[u_j (\widetilde b(s))]\int_{\gamma_{s,i}}  
\mathscr{C}_j(\widetilde b(s))\partial_\lambda[\widehat y_i(\widetilde b(s),\lambda)]
e^{-\frac{\lambda}{z}}{\mathrm d}\lambda.
\]
Noting $\partial_\lambda[\widehat{y}_i(\widetilde b(s),\lambda)e^{-\frac{\lambda}{z}}]=\partial_\lambda[\widehat y_i(\widetilde b(s),\lambda)]
e^{-\frac{\lambda}{z}}-\frac1z\widehat{y}_i(\widetilde b(s),\lambda)e^{-\frac{\lambda}{z}}$, we see that 
\begin{align*}
   (i)+(v)=&-\frac1z\sum_{j=1}^N\partial_s[u_j(\widetilde b(s))]\int_{\gamma_{s,i}}\mathscr C_j(\widetilde b(s))\partial_\lambda[\widehat y_i(\widetilde b(s),\lambda)e^{-\frac{\lambda}{z}}]\,{\mathrm d}\lambda\\
    =&-\frac1z\sum_{j=1}^N\partial_s[u_j(\widetilde b(s))]\left(\mathscr C_j(\widetilde b(s))\widehat y_i(\widetilde b(s),\gamma_{s,i}(t))e^{-\frac{\gamma_{s,i}(t)}{z}}\right)\bigg|^{t=+\infty}_{t=0}
\end{align*}
Now the term for $t=+\infty$ vanishes because $\widehat y_i(\widetilde b(s),\lambda)$ has moderate growth for $\lambda$ near $\lambda=\infty$ and $\operatorname{Re}(\frac{e^{\mathbf{i}\phi_i}}{z})>0$. So from $\gamma_{s,i}(0)=u_i(\widetilde b(s))$, we obtain
\begin{align*}
(i)+(v)=&\frac1z\sum_{j=1}^N\partial_s[u_j(\widetilde b(s))]\mathscr C_j(\widetilde b(s))\Psi_i(\widetilde b(s))e^{-\frac{u_i(\widetilde b(s))}{z}}\\
=&\frac1z\partial_s[u_i(\widetilde b(s))]\mathscr C_i(\widetilde b(s))\Psi_i(\widetilde b(s))e^{-\frac{u_i(\widetilde b(s))}{z}}.
\end{align*}
For $(ii),(iii),(iv)$, note that the sum of their integrands is equal to
\[
\partial_t[\widehat y_i(\widetilde b(s),\gamma_{s,i}(t))e^{-\frac{\gamma_{s,i}(t)}{z}}\partial_s\gamma_{s,i}(t)],
\]
and we get
\begin{align*}
    (ii)+(iii)+(iv)&=\frac1z\left(\widehat y_i(\widetilde b(s),\gamma_{s,i}(t))e^{-\frac{\gamma_{s,i}(t)}{z}}\partial_s\gamma_{s,i}(t)\right)\Bigg|^{t=+\infty}_{t=0}\\
    &=-\frac1z\Psi_i(\widetilde b(s))e^{-\frac{u_i(\widetilde b(s))}{z}}\partial_s[u_i(\widetilde b(s))].
\end{align*}
Therefore $(F^*\nabla)_{\partial_s}I_i(s,z)=0$.
Similarly, we have $(F^*\nabla)_{\partial_z}I_i(s,z)=0$. This proves $(F^*\nabla)I_i=0$.

So we see that $I_i(s,z)=\mathcal Z^K(E_i)(\widetilde b(s),z)$. In particular, let $\mathbf w'=\widetilde b(1)$, and we have
\[
\mathcal Z^K(E_i)(\mathbf w',z)
=
\frac1z\int_{L(u_i(\mathbf w'),\phi_i)}
\widehat y_i(\mathbf w',\lambda)
e^{-\lambda/z}\,\mathrm{d}\lambda.
\]
By slightly varying the contour and applying Watson's lemma (see \cite[Lemma 6.2.2]{AF}), we see that there exist phases $\phi'_1>\dots>\phi'_N$ with $\phi'_1-\phi'_N<2\pi$, such that
\begin{align*} 
         \mathcal{Z}^K(E_i)(\mathbf w',z) e^{\frac{u_i(\mathbf{w}')}{z}} \rightarrow \Psi_{i}(\mathbf{w}'), ~\text{as }  z \rightarrow 0 ~\text{with } |\arg z -\phi'_i| < \frac{\pi}{2}+\epsilon,
     \end{align*}
and the half lines $L(u_i(\mathbf w'),\phi'_i)$ are pairwise disjoint. Now the conclusion of the theorem follows from \cite[Theorem 3.8]{HuKe}.
\end{proof}
\begin{remark}
The argument of ``moving the quantum parameter in $\mathcal B$" was also used by Iritani in \cite[Proof of Theorem 7.33]{Iri20} to establish Orlov-type decompositions of quantum $D$-modules in the toric setting.
\end{remark}

\section{Mirror symmetry for Milnor hypersurfaces} \label{MS_Mil}
In this section, we prove   mirror symmetry for  Milnor hypersurfaces at the level of  ring isomorphisms with respect to Givental's mirror Landau-Ginzburg models.

Throughout the rest of this paper, we fix integers $n, m$ with $n\geq m\ge 2$. The (smooth) \textbf{Milnor hypersurface} $H_{n,m}$, up to isomorphism,  is the   hypersurface in $\mathbb{P}^{n} \times \mathbb{P}^{m}$ of degree $(1,1)$ defined by  $$\varsigma_0\hat\varsigma_0+\varsigma_1\hat\varsigma_1+\cdots+\varsigma_{m}\hat\varsigma_{m}=0 $$
in homogeneous coordinates $([\varsigma_0:  \cdots : \varsigma_{n}], [\hat\varsigma_0: \cdots: \hat\varsigma_{m}])$.  
Denote $n_1:=n$ and $n_2:=m$. For $i\in \{1, 2\}$, let $p_i: H_{n,m} \rightarrow \mathbb{P}^{n_i}$ be the natural projection  to $\bbP^{n_i}$.
Note that $H_{n, m}$ is a Fano manifold
 with  first Chern class
\begin{equation}\label{c1eqn}
    c_1(T_{H_{n, m}})=nh_1+mh_2, \quad\mbox{where}\quad     h_i:=p_i^*(c_1(\mathcal{O}_{\bbP^{n_i}}(1))).
\end{equation}
Moreover, $\{h_1, h_2\}$ form a nef basis of $H^{2}(H_{n,m}, \bbZ)$, and 
 we have    ${\rm Eff}(H_{n, m})=\mathbb{Z}_{\geq 0}\beta_1+\mathbb{Z}_{\geq 0}\beta_2\subset H_2(H_{n, m}, \bbZ)$, where   
 $\int_{\beta_a}h_{b}=\delta_{a,b}$ for $1\leq a,b \leq 2.$

 A ring presentation of the small quantum $K$-theory of $H_{n, m}$ was obtained in \cite[Theorem 1.5]{KLS}, by studying    Givental's $K$-theoretic $J$-function of $H_{n, m}$. By using parallel cohomological level arguments, we obtain the following result, whose proof is given in   Appendix A for completeness.

 \begin{prop} \label{QH_Mil}
 The small quantum cohomology ring  
$ QH^{*}(H_{n,m})$ is canonically isomorphic to the quotient ring $\mathbb{C}[h_1,h_2,q_1,q_2]/(f_{1}^{q},f_{2}^{q})
$, where
\begin{align*}
  f_1^{q}(h_1,h_2,q_1,q_2)&=h_2^{m+1}-q_2(h_1+h_2),\\
  f_2^{q}(h_1,h_2,q_1,q_2)&=\sum_{k=0}^{m}(-1)^{k}h_1^{n-k}h_2^{k}-q_1-(-1)^{m}q_2 h_1^{n-m}.   
\end{align*}
\end{prop} 

Following \cite{Giv98},  we consider the mirror Landau-Ginzburg model $(Z, f)$ 
for Fano hypersurfaces in a Fano toric manifold. In our case,    
 $$Z:=\big\{((x_i)_i,(y_j)_j, q_1, q_2)\in (\mathbb{C}^{\times})^{n+m}\times (\mathbb{C}^{\times})^2\mid\frac{q_1}{x_1\cdots x_{n}}+\frac{q_2}{y_1y_2\cdots y_{m}}=1\big\},$$ and the mirror superpotential $f: Z\to \mathbb{C}$ is defined by $$f:=x_1+x_2+\cdots+x_{n}+y_1+y_2+\cdots+y_{m}.$$ 
With the substitution $x_{n}=\frac{q_1(y_1\cdots y_{m})}{x_1 \cdots x_{n-1}(y_1\cdots y_{m}-q_2)}$, we can identify $Z$ with the complement of the hypersurface $\{y_1 y_2\cdots y_m-q_2=0\}$ in 
$(\mathbb{C}^{\times})^{n+m-1}\times (\mathbb{C}^{\times})^2$. 
Then the Jacobi ring ${\rm Jac}(f)$ of $f$ is defined by 
\begin{align*}  
    {\rm Jac}(f):= \mathbb{C}[x_{1}^{\pm 1},\cdots, x_{n-1}^{\pm 1},y_{1}^{\pm 1},\cdots, y_{m}^{\pm 1},q_1^{\pm 1},q_2^{\pm 1},(y_1\cdots y_{m}-q_2)^{-1}]\Big/ I,
\end{align*} 
 with the quotient ideal given by     $$I=\big(x_1 \frac{\partial f}{\partial x_1},\cdots,x_{n-1}\frac{\partial f}{\partial x_{n-1}},y_1 \frac{\partial f}{\partial y_1},\cdots,y_{m}\frac{\partial f}{\partial y_{m}}\big).$$

\begin{thm} \label{MS}
The map $(h_1, h_2, q_1, q_2)\mapsto ([x_1], [y_1], [q_1], [q_2])$ gives a ring isomorphism in both of the following cases.
  \begin{enumerate} 
    \item $\Phi: QH^{*}(H_{n,m})[q_1^{-1}, q_2^{-1}]\overset{\cong}{\longrightarrow} {\rm Jac}(f) \mbox{ if } n>m\geq 2;$
     \item $\Phi: QH^{*}(H_{n,m})[q_1^{-1}, q_2^{-1}, h_1^{-1}, h_2^{-1}]\overset{\cong}{\longrightarrow} {\rm Jac}(f) \mbox{ if } n=m\geq 2$.
\end{enumerate}  
 Furthermore, $\Phi(c_1(T_{H_{n, m}}))=[f]$.  
 
\end{thm}
\begin{proof}
By direct computation of  $x_i\frac{\partial f}{\partial x_i}$ and $y_j\frac{\partial f}{\partial y_j}$,  in ${\rm Jac}(f)$ we have   
 \begin{align*} 
    &x_1=x_2=\cdots=x_{n-1}=\frac{q_1(y_1\cdots y_{m})}{x_1 \cdots x_{n-1}(y_1\cdots y_{m}-q_2)}, \\
    &y_1=y_2=\cdots=y_{m}= \frac{q_1 q_2(y_1\cdots y_{m})}{x_1 \cdots x_{n-1}(y_1\cdots y_{m}-q_2)^{2}}=\frac{x_1 q_2}{ y_1\cdots y_{m}-q_2}.  
 \end{align*} 
\noindent It follows immediately that 
\begin{align*} 
    {\rm Jac}(f)&\cong \mathbb{C}[x_{1}^{\pm 1},y_{1}^{\pm1},q_1^{\pm1},q_2^{\pm1},(y_1^{m}-q_2)^{-1}]\big/\big(x_1-\frac{q_1 y_1^{m}}{x_1^{n-1}(y_1^{m}-q_2)},y_1-\frac{ x_1q_2 }{y_1^{m}-q_2}\big)\\
    &\cong \mathbb{C}[x_{1}^{\pm 1},y_{1}^{\pm1},q_1^{\pm1},q_2^{\pm1},(y_1^{m}-q_2)^{-1}]\big/\big(x_1^{n}\frac{y_1^{m}-q_2}{y_1^{m}}-q_1, y_1(y_1^{m}-q_2)- x_1q_2\big)\\
   &\cong \mathbb{C}[x_{1}^{\pm 1},y_{1}^{\pm1},q_1^{\pm1},q_2^{\pm1}]\big/\big(x_1^{n}\frac{y_1^{m}-q_2}{y_1^{m}}-q_1, y_1(y_1^{m}-q_2)- x_1q_2\big).
\end{align*}
The last isomorphism holds because $y_1^m-q_2$ is already a unit in the quotient ring on the right hand side, following from the second relation. 
 Repeatedly using  relation $y_1(y_1^{m}-q_2)=q_2 x_1$ in the quotient ring, we have 
\begin{align*}
  x_1^{n}\frac{y_1^{m}-q_2}{y_1^{m}}-q_1 
    =& x_1^{n}-x_1^{n-1}\frac{x_1 q_2}{y_1^{m}}-q_1 \\
    =&x_1^{n}-x_1^{n-1}\frac{y_1(y_1^{m}-q_2)}{y_1^{m}}-q_1 \\
    =&x_1^{n}-x_1^{n-1}y_1+ x_1^{n-2}\frac{x_1q_2}{y_1^{m-1}}-q_1 \\
    =&x_1^{n}-x_1^{n-1}y_1+ x_1^{n-2}y_1^{2}+\cdots +(-1)^{m} x_1^{n-m} y_1^{m} -(-1)^{m}x_{1}^{n-m} q_2-q_1.
\end{align*}
  Therefore 
 \begin{align*}
    {\rm Jac}(f)\cong \mathbb{C}[x_{1}^{\pm 1},y_{1}^{\pm1},q_1^{\pm1},q_2^{\pm1}]\big/\big(f_{1}^{q}(x_1,y_1,q_1,q_2),f_{2}^{q}(x_1,y_1,q_1,q_2)\big).
\end{align*}
Combining the above isomorphism  with Proposition \ref{QH_Mil}, we obtain
a ring isomorphism  
$QH^{*}(H_{n,m})[q_1^{-1}, q_2^{-1}, h_1^{-1}, h_2^{-1}]\overset{\cong}{\longrightarrow} {\rm Jac}(f)$ induced by $h_1\mapsto x_1, h_2\mapsto y_1$, $q_1\mapsto q_1$ and $q_2\mapsto q_2$. 
Consequently, 
this map sends $c_1(T_{H_{n,m}})=nh_1+mh_2$ to $[f]$, since all $[x_i]$ (resp. $[y_j]$) are the same in the ${\rm Jac}(f)$.

Assume $n>m$ now. Since we can write  $f_2^q(h_1, h_2, q_1, q_2)= h_1^{n-m}u-q_1$ and $q_1$ is a unit,
   $h_1$ and $u$ are both units in   $QH^{*}(H_{n,m})[q_1^{-1}, q_2^{-1}]$. In this quotient ring,  
$(h_1+h_2)u
=
h_1^{m+1}+(-1)^m h_2^{m+1}
-(-1)^m q_2(h_1+h_2)
=
h_1^{m+1}.$ 
This implies that
$h_1+h_2$ is a unit, and  so is $h_2$ by the relation $f_1^q$. Hence, the localizations $h_1^{-1}, h_2^{-1}$ are redundant when $n>m$. 
\end{proof}
Denote by $\mbox{Spec}(\hat c_1)=\{\{\lambda_1, \cdots, \lambda_N\}\}$ the multiset of eigenvalues of the linear operator $c_1(T_{H_{n,m}})*$ on $QH^*(H_{n, m})[q_1^{-1}, q_2^{-1}]$. For instance,  $\{\{1, 2, 2,2\}\}$ is read off as eigenvalue $1$ of multiplicity one and eigenvalue  $2$ of multiplicity three. Similarly, we denote by $\mbox{CV}(f)$  the multiset of the critical values of $f$. Denote by $\dot\cup$ the union of multisets.  
\begin{cor}\label{cor:eigenCV}
     Let $n\geq m\geq 2$ and $q_1, q_2\neq 0$. 
     If both $n=m$ and $q_1=(-1)^{n+1}q_2$ hold, then {\upshape $\mbox{Spec}(\hat c_1)=\mbox{CV}(f)\dot\cup \{\{0, \cdots, 0\}\}$}, where there are $n$ copies of $0$. Otherwise, {\upshape $\mbox{Spec}(\hat c_1)=\mbox{CV}(f)$}. 
\end{cor}

\begin{proof} For $q_1, q_2\neq 0$, we note that the critical values of $f$ are simply the eigenvalues of the linear operator $[f]$ on $\mbox{Jac}(f)$, and that 
   $\mbox{Spec}(\hat c_1)$ coincides with the multiset $\mbox{EV}(nh_1+mh_2)$ of the values of $nh_1+mh_2$ at the
   spectrum of  $\mathbb{C}[h_1, h_2]/(f_1^q, f_2^q)$ by \cref{QH_Mil}.
   Thus by \cref{MS}, we have $\mbox{CV}(f)=\mbox{Spec}([f])=\mbox{Spec}(\hat c_1)$, whenever either (a) $n>m$ or (b) $n=m$ and  the
   spectrum of $\mathbb{C}[h_1, h_2]/(f_1^q, f_2^q)$ contains no points in the variety $h_1h_2=0$.

 Assume $n=m$. If $h_2=0$,  the equations    $f_1^q=0$ and $f_2^q=0$ have solutions only if $q_1+(-1)^nq_2=0$; when this holds, we obtain the nonreduced point  $(0, 0)$ with multiplicity $n$.
  If $h_2\neq 0$ and $h_1=0$, $f_1^q- (-1)^nh_2f_2^q=(-1)^n q_1 h_2\neq 0$, so that there are no solutions of the equations $f_1^q=f_2^q=0$.  Hence,  $\mbox{Spec}(\hat c_1)\setminus \mbox{CV}(f)$ is the multiset of $0$ with multiplicity $n$,  if $q_1=(-1)^{n+1}q_2$, or the empty multiset otherwise.
\end{proof}

\begin{remark}\label{rmk:schu}
Via the    Pl\"ucker embedding, the two-step flag variety  $F\ell_{1,n;n+1}$ is a degree $(1, 1)$-hypersurface of 
 $\mathbb{P}^{n}\times \mathbb{P}^{n}$, defined by 
  $\sum_{j=0}^{n}(-1)^j \varsigma_j \hat\varsigma_{j}=0$. Hence it is isomorphic to $H_{n, n}$   by a simple coordinate change. Moreover, any smooth Schubert subvariety in $F\ell_{1,n;n+1}$ is either a Milnor hypersurface (where we allow $m=1$) or a product of projective spaces. 
  
  Note that $F\ell_{1,n;n+1}$ admits a toric degeneration, and so does $H_{n, m}$ by restriction. Hence, we can also construct a Landau-Ginzburg model of $H_{n, m}$, in analogy with the constructions in   \cite{Giv97, BCFKS, HKLS}. The resulting   superpotential is a Laurent polynomial defined on $(\mathbb{C}^\times)^{n+m-1}\times (\mathbb{C}^\times)^2$, extending $(Z, f)$, and  we expect that it can be further extended to the Lie theoretic Landau-Ginzburg model constructed in \cite{LRY}.
\end{remark}

\section{Gamma-I  for Milnor hypersurfaces}
The notion of property Gamma-I for general Fano manifolds  was introduced in \cite{HKLS2}, based on the initial study \cite{GHIKLS} of Gamma conjecture I over the (complexified) K\"ahler moduli space. In this section, we show that property Gamma-I holds for $H_{n, m}$ in $V_{\mathbb{R}}:=H^2(H_{n, m};\mathbb{R})$.
Consequently, Galkin-Golyshev-Iritani's  Gamma Conjecture I holds for  $H_{n, m}$.

 \subsection{Property Gamma-I} \label{Gamma I}
Let $X=H_{n, m}$. For $\eta\in H^2(X)$, we can write $\eta=t_1h_1+t_2h_2$ and $q_i=e^{t_i},i=1,2$. 
 Denote by $\hat c_1(\eta)$ the operator  in $\mathrm{End}(H^*(X))$ defined by $\alpha\mapsto (c_1(T_X)*_q \alpha)|_{(q_1, q_2)=(e^{t_1}, e^{t_2})}$.
\begin{defn} \label{cond (SR)}
 We say that $X$ satisfies condition ${\rm (SR)}$ at $\eta$, if   $\hat{c}_1(\eta)$ 
 has a simple  rightmost eigenvalue $u_{\eta}$, namely  for any other eigenvalue $u \in {\rm Spec}  ~\hat{c}_1(\eta)$, we have ${\rm Re}~ u < {\rm Re}~ u_\eta.$
 By the $(\mathrm{SR})$-region of $X$, we mean the subset  $\{\eta\in H^2(X) \mid X \mbox{ satisfies condition } (\mathrm{SR})\mbox{ at } \eta \}$ of $H^2(X)$.  
\end{defn}
  
Restricting to the small quantum cohomology, we have the (small) Dubrovin connection $\nabla^{\rm sm}$ on the trivial $H^*(X)$-bundle over $H^2(X)\times \bbP^1$, given by $\nabla^{\rm{sm}}=\mathrm{d}+\sum_{i=1}^{2}{1\over z}h_i*_q \mathrm{d}t_i-\left({1\over z} {c}_1(T_X)*_q-\mu\right){\mathrm{d}z\over z}.$ 
Similarly, we consider the  space $\mathcal{S}^{{\rm{sm}}}$ of $\widetilde{\nabla}^{\rm{sm}}$-flat sections of the trivial $H^*(X)$-bundle over $H^2(X)\times \widetilde{\mathbb{C}^\times}$.
 Suppose that $X$ satisfies condition $(\mathrm{SR})$  at $\eta\in H^2(X)$ with  simple  rightmost eigenvalue $u_{\eta}$ of $\hat{c}_1(\eta)$. Also consider  the   group homomorphism $\mathcal{Z}^{K, {\rm sm}}: K(X)\to \mathcal{S}^{\rm sm}$.  Denote by $\mathcal{A}(\eta)$   the $\bbC$-linear  space   of $\widetilde{\nabla}^{\rm{sm}}$-flat sections $y(\eta,z)$ such that $e^{{u_{\eta}\over z}}y(\eta,z)$ has moderate growth as $z\to0$ along the positive real axis (i.e. $\arg z=0$). Then $\dim _{\bbC}{\mathcal{A}}(\eta)=1$ by    \cite[Remark 3.1.9]{GGI}. 

\begin{defn}[Property Gamma-I] \label{PGammaI}
We say that \emph{$X$ satisfies Gamma-I at $\eta\in H^2(X)$}, if   $X$ satisfies condition $(\mathrm{SR})$ at $\eta$ and  $\mathcal{Z}^{K,{\rm{sm}}}(\cO_X)(\eta,z)\in\mathcal{A}(\eta)$.
\end{defn}
\begin{remark}
    The special case when $\emph{X}$ satisfies Gamma-I at $\eta=\mathbf{0}$ is a variant of Gamma conjecture I of Galkin-Golyshev-Iritani.    
\end{remark}

\subsection{ ${\rm (SR)}$-region of   $H_{n,m}$} \label{Gamma_I_Milhyp}
Recall  $n\geq m\geq 2$ and $V_{\mathbb{R}}=H^{2}(H_{n,m};\mathbb{R}).$

\begin{lemma}   \label{root_re}
Let $q_1,q_2\in \mathbb{R}_{>0}$. Then there is a real simple root $r_0$  of  
\begin{align*} 
    F(y):= q_2 y^{n-m} ( \frac{y^{m}}{q_2}-1)^{n+1}-q_1
\end{align*}
with $r_0>q_2^{\frac{1}{m}}$. Moreover, for any   root $z$ of $F(y)$, we have $|z| \leq r_0$, with equality   only if 
$z=r_0 e^{\mathbf{i} \theta}$ 
and $\theta\in {2\pi \over  {\rm g.c.d.}(n,m)}\mathbb{Z}$.
\end{lemma}

\begin{proof}
It follows from   the derivative of $F(y)$,
      $$F^{'}(y)=q_2y^{n-m-1} (\frac{y^{m}}{q_2}-1)^{n} (n(m+1)\frac{y^{m}}{q_2}-(n-m)), $$  that $F(y)$  is increasing on $(q_2^{\frac{1}{m}},+\infty)$. Since $F(q_2^{\frac{1}{m}})=-q_1<0$ and $F(+\infty)=+\infty$,   there exists a  real root $r_0$ of $F(y)$ with $r_0>q_2^{\frac{1}{m}}$. Since ${r_0^m\over q_2}>1$, we have $F'(r_0)>0$; consequently,  $r_0$ is a simple root.

Assume that there is a root $z$ of $F(y)$ with $|z|\geq  r_0$. Since   $|\frac{z^{m}}{q_2}-1| \geq \frac{|z|^{m}}{q_2}-1\geq  {r_0^m\over q_2}-1>0$,  we have
\begin{align} \label{inequ1}
    q_1 = q_2|z|^{n-m} |\frac{z^{m}}{q_2}-1|^{n+1} \geq q_2 |z|^{n-m} (\frac{|z|^{m}}{q_2}-1)^{n+1}=F(|z|)+q_1.
\end{align}
Thus $F(|z|)\leq 0=F(r_0)$, implying that $|z|=r_0$ and    the equality in \cref{inequ1} holds, namely $|\frac{z^{m}}{q_2}-1| = \frac{|z|^{m}}{q_2}-1.$
Consequently,   $z^{m}\in \mathbb{R}_{>0}$,  implying   $z=r_0e^{\mathbf{i} \theta }$ with $e^{\mathbf{i} m\theta }=1$. Hence, we have $0=F(r_0e^{\mathbf{i} \theta })=q_2 r_0^{n-m}e^{\mathbf{i} (n-m)\theta}({r_0^m\over q_2}-1)^{n+1}-q_1$.  Comparing it with $F(r_0)=0$, we obtain $e^{\mathbf{i} (n-m)\theta}=1$. Hence, $e^{\mathbf{i}\theta\cdot {\rm g.c.d.}(n, m)}=1$. 
\end{proof}

\begin{prop}  \label{cond_* for Milhyp}
 For any $\eta \in V_{\mathbb{R}}$,   $H_{n, m}$  satisfies condition $(\mbox{SR})$  at $\eta$. 
\end{prop}

\begin{proof}
   By Corollary \ref{cor:eigenCV}, 
it suffices to  analyze   the critical values of the superpotential $f$, which are given by the values of  $E(y)=n\frac{y^{m+1}}{q_2}+(m-n)y$ at the roots of $F(y)$. (Indeed, the only possibly missing eigenvalue  is $0$, which also satisfies $0<E(r_0)$.) 
Note that $q_i=e^{t_i}>0$ and take the real simple root $r_0 \in (q_2^{\frac{1}{m}},+\infty)$ of 
$F(y)$ given in \cref{root_re}, and let  $z$ also be a root of   $F(y)$.

By    Lemma \ref{root_re}, we have    $|z| \leq r_0$, and if  $|z|=r_0$ holds, then  $z=r_0e^{\mathbf{i} \theta}$  with $\theta\in {2\pi\over {\rm g.c.d.}(n, m)}\mathbb{Z}$. In this case,
\begin{equation}\label{eqn:propO2} E(z)=ne^{\mathbf{i} \theta} \frac{r_0^{m+1}}{q_2}+(m-n) r_0e^{\mathbf{i}\theta}=E(r_0)e^{\mathbf{i} \theta}.
  \end{equation}
  
  Now assume $0<|z|<r_0$. Since $r_0$ and $z$ are both   roots of   $F(y),$      we have $\frac{r_0^{m}}{q_2}-1=(\frac{q_1}{q_2})^{\frac{1}{n+1}}(r_0)^{-\frac{n-m}{n+1}}$ and $\frac{z^{m}}{q_2}-1=(\frac{q_1}{q_2})^{\frac{1}{n+1}}(z)^{-\frac{n-m}{n+1}}w$ with $w^{n+1}=1$.
          This implies $E(z)=z\big(n(\frac{z^{m}}{q_2}-1)+m\big)=z\big(n (\frac{q_1}{q_2})^{\frac{1}{n+1}}(z)^{-\frac{n-m}{n+1}}w +m\big).$ Hence,
          \begin{equation}\label{eqn:proO1}
               |E(z)|\leq  n(\frac{q_1}{q_2})^{\frac{1}{n+1}} |z|^{\frac{m+1}{n+1}} +m|z| < n (\frac{q_1}{q_2})^{\frac{1}{n+1}}r_0^{\frac{m+1}{n+1}} +m r_0= E(r_0),
          \end{equation}
       where  the last equality follows from  $ (\frac{q_1}{q_2})^{\frac{1}{n+1}}r_0^{\frac{m+1}{n+1}}=r_0(\frac{r_0^{m}}{q_2}-1)$. 
      
    Thus for any $z\neq r_0$, we have   ${\rm Re}~ E(z) < {\rm Re}~ E(r_0)$.        
\end{proof}

Note that the Fano index of $H_{n, m}$ equals ${\rm g.c.d.} (n, m)$. The following is a direct consequence of $\eqref{eqn:proO1}$ and $\eqref{eqn:propO2}$.
\begin{cor}\label{cor:conjO}
Property $\mathcal{O}$ holds for $H_{n, m}$. That is,  the spectral radius $\rho$ of $\hat c_1(\mathbf{0})$ is a simple eigenvalue, and if $u$ is an eigenvalue of $\hat c_1(\mathbf{0})$ of modulus $\rho$, then $({u\over \rho})^{\rm g.c.d.(n,m)}=1$.     
\end{cor}

\subsection{Gamma-I for $H_{n, m}$}  \label{Gamma_I_Milhyp2}
Recall that  $(Z,f)$ is Givental's Landau-Ginzburg model  for $H_{n,m}$. Denote by  $Z_{\eta}$ the fiber of $Z\to (\mathbb{C}^\times)^2$ at $(q_1, q_2)=(e^{t_1}, e^{t_2})$, where  $\eta=t_1h_1+t_2h_2 \in H^{2}(H_{n,m})$.  Denote $Z_{\rm  real,\eta}:=Z_\eta\cap (\mathbb{R}_{>0}^{n+m}\times \mathbb{R}_{>0}^2)$.
  In  \cite[Theorem 7.5]{GaIr}, Galkin and Iritani showed that 
the (original) Gamma conjecture I holds for $H_{n, m}$, provided that Condition \ref{condition} holds at $\eta=0$. 
  \begin{cond}{\rm  \cite[Condition 7.2]{GaIr}}  \label{condition}
       Let $\eta\in V_{\mathbb{R}}$. We have: 
   
      (a) $Z_{\eta}$ is a smooth hypersurface in $(\mathbb{C}^{\times})^{n+m}$;

  (b) $f|_{Z_{\rm  real,\eta}}: Z_{\rm  real,\eta} \rightarrow \mathbb{R}$ attains the global minimum at a unique critical point $x_{\rm con}\in Z_{\rm  real,\eta}$, which is called the conifold point;

  (c) the conifold point $x_{\rm con}$ is a non-degenerate critical point of $f$;

  (d) the spectral radius of   $c_1(T_{H_{n,m}})\ast_{\eta } $   coincides with $T_{\rm con}:=f(x_{\rm con})$;

  (e) $T_{\rm con}$ is an eigenvalue of $c_1(T_{H_{n,m}})\ast_{\eta} $ with multiplicity one.
  \end{cond}

\begin{lemma} \label{conifold}
    For $\eta \in V_{\mathbb{R}}$, parts (b) and (c) of Condition \ref{condition} hold for $H_{n, m}$.
\end{lemma}
\begin{proof}
      By the inequality of arithmetic and geometric means, we have 
     \begin{align*}
         f|_{Z_{\rm  real,\eta}}=&x_1+x_2+\cdots+x_{n-1}+x_n+y_1+y_2+\cdots+y_{m} \\
         \geq & (n-1)(x_1 \cdots x_{n-1})^{\frac{1}{n-1}}+x_n+m(y_1\cdots y_{m})^{\frac{1}{m}},
     \end{align*}
    with equality   if and only if   $x_1=x_2=\cdots=x_{n-1}$ and $y_1=y_2=\cdots=y_m$ hold.
    By the definition of $Z_{\rm  real,\eta}$, we have  $x_n=\frac{q_1 y_1\cdots y_{m}}{x_1 \cdots x_{n-1}(y_1\cdots y_{m}-q_2)}>0$, implying $y_1\cdots y_{m}>q_2$. 
   Consider the function $g:\mathbb{R}_{>0}\times (q_2, +\infty)\to \mathbb{R}_{>0}$, 
   $$g(c,d):=(n-1)c^{\frac{1}{n-1}}+ \frac{q_1d}{c(d-q_2)}+m d^{\frac{1}{m}} {\rm ~s.t.~} c>0, ~d>q_2. $$
    Clearly,  $g(c,d) \rightarrow +\infty$ whenever $c$ or $d$ tends to infinity or $(c, d)$ approaches the boundary of the defining domain. Thus $g(c,d)$ attains the global minimum at some $(c_0,d_0)$ with   $\frac{\partial g}{\partial c}(c_0,d_0)=0,\frac{\partial g}{\partial d}(c_0,d_0)=0$.  It follows that  $c_0^{\frac{1}{n-1}} q_2=(d_0-q_2)d_0^{\frac{1}{m}}$ and 
    $  q_2 (d_0^{\frac{1}{m}})^{n-m}(\frac{(d_0^{\frac{1}{m}})^{m}}{q_2} -1)^{n+1}-q_1=0.$
    In particular,  $d_0^{\frac{1}{m}}$ is a real root of  $F(y)$ and $d_0^{\frac{1}{m}}> q_2^{1\over m}$. Thus   $d_0^{\frac{1}{m}}$ equals    $r_0$ in Lemma  \ref{root_re}.
   Consequently, $f|_{Z_{\rm  real,\eta}}$ attains the global minimum at the unique  $$x_{\rm con}=(c_0^{\frac{1}{n-1}},c_0^{\frac{1}{n-1}},\cdots,c_0^{\frac{1}{n-1}}, \frac{q_1 d_0}{c_0(d_0-q_2)}, d_0^{\frac{1}{m}},d_0^{\frac{1}{m}},\cdots,d_0^{\frac{1}{m}}, q_1, q_2) \in Z_{\rm real,\eta},$$ 
which is a critical point of $f$.

    To show that $x_{\rm con}$ is   nondegenerate, we  consider the (logarithmic) Hessian matrix 
  \begin{align*}
    {\rm Hess}=
        \begin{pmatrix}
          A & D \\
          D^{t} & C
   \end{pmatrix},
  \end{align*}
where $A=(\frac{\partial^{2} f}{ \partial \log x_i\partial \log x_j})$, $D=(\frac{\partial^{2} f}{ \partial \log x_i\partial \log y_j})$, $C=(\frac{\partial^{2} f}{ \partial \log y_i\partial \log y_j})$ are matrices of sizes  $(n-1)\times (n-1), (n-1)\times m, m \times m $, respectively.  Denote $c:=x_1x_2 \cdots x_{n-1}$ and $d:=y_1 y_2\cdots y_m$.    By direct computation, we have
\begin{align*}
    A=&{\rm diag}(x_1,\cdots,x_{n-1})+\frac{q_1d}{c(d-q_2)} \vec{u} (\vec{u})^{t} ;  \\
    D=&\frac{q_1q_2d}{c(d-q_2)^{2}} \vec{u} (\vec{v})^{t}; \\
    C=&{\rm diag}(y_1,\cdots,y_{m})+\frac{q_1 q_2 (d^2+q_2 d)}{c(d-q_2)^{3}}  \vec{v} (\vec{v})^{t},
\end{align*}
 where $\vec{u}$ (resp. $\vec{v}$) is an all‑ones column vector of dimension 
 $(n-1)$ (resp. $m$).  Denote $k_1:=\frac{q_1 d}{c(d-q_2)}$ and $k_2:=\frac{q_1q_2d}{c(d-q_2)^{2}}$.   By the Sherman–Morrison formula in linear algebra, $A$ is invertible if and only if 
 \begin{align} \label{invert_A}
     1+k_1 (\vec{u})^{t} ~{\rm diag}(x_1^{-1},\cdots,x_{n-1}^{-1}) \vec{u} \neq 0;
 \end{align}
 
this nonvanishing condition obviously holds at $x_{\rm con}$. Thus  $\det(A) \neq 0$ at $x_{\rm con}$, and consequently we have
$\det ({\rm Hess})=\det(A)\det(C-D^{t}A^{-1}D),$ where 
\begin{align*}
    C-D^{t}A^{-1}D={\rm diag}(y_1,\cdots,y_{m})+\Big(\frac{q_1 q_2 (d^2+q_2 d)}{c(d-q_2)^{3}}-\frac{ k_2^{2}(\sum_{i=1}^{n-1}x_i^{-1})}{1+k_1(\sum_{i=1}^{n-1}x_i^{-1})}\Big)  \vec{v} (\vec{v})^{t}.
\end{align*}
 It  follows from
 $$d \cdot \Big(1+(\sum_{j=1}^{m} y_j^{-1} ) \frac{q_1q_2(d^{2}+q_2d)c(d-q_2)+q_1^{2}q_2d^{3}(\sum_{i=1}^{n-1}x_i^{-1})}{c(d-q_2)^{3}(c(d-q_2)+q_1d(\sum_{i=1}^{n-1}x_i^{-1}))}\Big)>0$$ that 
  $\det (C-D^{t}A^{-1}D ) \neq 0 $ at   $x_{\rm con}$. Hence, $\det ({\rm Hess})\neq 0$ 
  at  $x_{\rm con}$.
\end{proof}

\begin{thm} \label{Gamma-I}
 For any $\eta \in V_{\mathbb{R}}$, $H_{n,m}$ satisfies property Gamma-I at $\eta$. Consequently, Gamma conjecture I holds for $H_{n, m}$.
 \end{thm}

\begin{proof}
We first verify that Condition \ref{condition} holds for $H_{n, m}$. Indeed, 
part (a) follows immediately from the definition of $Z_{\eta}$.
Parts (b) and (c) hold by \cref{conifold}. Parts (d) and (e) follow from  \cref{cor:eigenCV}
and (the proof of) \cref{cond_* for Milhyp}.

 Note that the   morphism  $H^{*}(\mathbb{P}^{n} \times \mathbb{P}^{m})\to H^{*}(H_{n,m})$ induced from the natural inclusion of $H_{n, m}$ is surjective, and that property $\mathcal{O}$ holds for $H_{n, m}$ by \cref{cor:conjO}. The argument for the rest is the same as that   for the case $\mathbf q=\mathbf 1$ given in \cite[Theorem 7.5]{GaIr}. Here we include it for completeness.  
 
Let $g_1=\frac{q_1}{x_1\cdots x_{n}}+\frac{q_2}{y_1y_2\cdots y_{m}}$, and  $\omega_{Z,\eta}=\frac{\frac{d x_1}{x_1} \wedge \cdots \wedge \frac{d x_n}{x_n} \wedge \frac{d y_1}{y_1} \wedge \cdots \wedge \frac{d y_m}{y_m}}{dg_{1}}$  denote the holomorphic volume form of $Z_{\eta}$,   defined  by taking the residue form on $g_1=1$ associated with fixed $\eta \in V_{\mathbb{R}}$.
 In \cite{Iri11}, it is proved that   
$$ (2\pi z)^{\frac{M}{2}} (1, \mathcal{Z}^{K, {\rm sm}}(\mathcal{O})(\eta,z))=\int_{Z_{{\rm real},\eta}} e^{-\frac{f}{z}} \omega_{Z,\eta},\quad \mbox{ where } M=\dim_{\mathbb{C}} H_{n,m}. ~~  $$

By differentiating this equality in the $\eta$-direction, (suppressing signs for simplicity,)  we obtain 
 $$ (2\pi z)^{\frac{M }{2}} (\nabla_{z\frac{\partial}{\partial t_{i_1}}} \cdots \nabla_{z\frac{\partial}{\partial t_{i_l}}} 1, \mathcal{Z}^{K, {\rm sm}}(\mathcal{O})(\eta,z))=z \frac{\partial}{\partial t_{i_1}} \cdots  z \frac{\partial}{\partial t_{i_l}} (\int_{Z_{{\rm real},\eta}} e^{-\frac{f}{z}} \omega_{Z,\eta}),  $$
 where $i_{j} \in \{1,2\}$ for $1\leq j \leq l$.
 Since $H^{*}(H_{n,m})$ is generated by $H^{2}(H_{n,m})$, we obtain the formula
 \begin{align} \label{int_O}
     (2\pi z)^{\frac{M}{2}} (\alpha, \mathcal{Z}^{K, {\rm sm}}(\mathcal{O})(\eta,z))=\int_{Z_{{\rm real},\eta}}  \varphi(x,y,z) e^{-\frac{f}{z}} \omega_{Z,\eta},
 \end{align}
for every $\alpha \in H^{*}(H_{n,m})$ associated with some function $\varphi(x,y,z)$. For $\eta \in V_{\mathbb{R}}$, the polynomial growth of $\varphi$, together with the equality \ref{int_O} and {\rm Condition  \ref{condition}}, implies that  $|| e^{\frac{T_{\rm con}}{z} } \mathcal{Z}^{K, {\rm sm}}(\mathcal{O})(\eta,z)||$   has moderate growth as $z\to0$ along the positive real axis.
\end{proof}

\section{Gamma conjecture II for Milnor hypersurfaces}  
In this section, we use a strategy similar to that in   \cite{HKLS2} to verify Gamma conjecture II for $H_{n, m}$. This involves property Gamma-I, Iritani's   {Galois action} and  the asymptotic behavior of eigenvalues of $c_1(T_{H_{n,m}}) \ast_{q}$.

\subsection{Basic properties}
Recall that $\mathcal{S}$ is the space of $\widetilde\nabla$-flat  sections.
\begin{prop}[{}\!\!\protect{\cite[Section 2.2]{Iri}}]
    For $\xi\in H^{2}(X,\mathbb{Z})$, the \textbf{Galois action} $G^\mathcal S(\xi)$ on $\mathcal S$ is defined by 
    $$G^{\mathcal{S}}(\xi): \mathcal{S} \longrightarrow \mathcal{S}, s(\tau,z) \mapsto s(\tau+2\pi\mathbf i\xi,z). $$
 The Galois action   preserves the lattice $\mathcal{S}_{\mathbb{Z}}$ and corresponds to  tensoring by the line bundle $L_{\xi}^{\vee}$ in $K(X):$
 \begin{align} \label{Galois_ac}
     \mathcal{Z}^{K}(L_{\xi}^{\vee} \otimes E)=G^{\mathcal{S}}(\xi)( \mathcal{Z}^{K}(E)).
 \end{align}\end{prop}
\noindent Here $L_{\xi}$ is  the line bundle with $c_1(L_{\xi})=\xi$.  The above important property is also called \textbf{monodromy-invariance}.

 Recall the natural projections $p_1: H_{n,m} \rightarrow \mathbb{P}^{n}$, $p_2:H_{n,m} \rightarrow \mathbb{P}^{m}$, and note that   $H_{n,m}$  is   a 
  $\mathbb{P}^{n-1}$-bundle over $\mathbb{P}^m$.  Recall $N=n(m+1)=\dim_\bbC H^*(H_{n,m}).$ 
By  \cite[Corollary 2.7]{Orl}, we have 
\begin{prop}[Orlov]\label{prop:FEC}
   For $1\leq t\leq N$, we write the Euclidean division $t-1=(m+1)a+b$ and then define $$F_t=\mathcal{O}(a-n+1, b-m):=p_1^{*}\mathcal{O}_{\mathbb{P}^{n}}(a-n+1)\otimes p_2^*\mathcal{O}_{\mathbb{P}^{m}}(b-m).$$ Then 
   $(F_1, F_2, \cdots, F_N)$ is a full exceptional collection in $\mathcal{D}^{b}(H_{n, m})$. 
\end{prop}

The next proposition is a special case of  \cite[Corollary 5.9]{Iri07}, which works for Fano manifolds with $H^*(X)$ generated by $H^2(X)$.
\begin{prop}\label{prop:convergeH}
   $QH_{\rm big}(H_{n, m})$  {converges near the large radius limit}.  
 \end{prop}

Again note that  $H_{n,m}$ is a 
  $\mathbb{P}^{n-1}$-bundle over $\mathbb{P}^m$.   The next proposition is a special case of    \cite[Corollary 1.8]{IrKo}. 
\begin{prop}\label{cor: ssH}
    $QH_{\rm big}(H_{n, m})$  is semisimple.
\end{prop} 

 \subsection{Gamma conjecture II for $H_{n, m}$}
 
The parameter region in $(\mathbb{C}^\times)^2$ considered below avoids the hypersurface defined by $q_1+(-1)^nq_2=0$. Thus by \cref{cor:eigenCV},  the eigenvalues of  $\hat c_1(\eta)$ (where $q=e^\eta$) coincide with the critical values of $f$; they are given by 
$n x_1+my_1=E(y_1)=n\frac{y_1^{m+1}}{q_2}+(m-n)y_1, $ with   $y_1$ being a root of   $F(y)=q_2 y^{n-m} ( \frac{y^{m}}{q_2}-1)^{n+1}-q_1$ as in Lemma \ref{root_re}.
 By Puiseux series expansion, the asymptotic behavior of $y_1$ is given by
\begin{align} \label{asym of cri_point}
y_1 =\kappa q_2^{\frac{1}{m+1}} + \frac{n+1}{N}\kappa^{1-m} q_2^{\frac{2}{m+1}} - \frac{n_0}{N}\kappa^{1-2m} q_2^{\frac{3}{m+1}} +O\!\big(q_2^{\frac{4}{m+1}}\big)
\end{align}
 as $q_2 \to 0$, where $N=n(m+1)$ and $$\kappa^{N}=q_1, \qquad n_0=\frac{(n+1)(mn+2m-2n-1)}{2N}.$$   Although $q_1$ varies along paths representing the Galois action, the following error bounds as $q_2\to 0$ are uniform in $q_1$ on the specified annulus.
By \cref{prop:convergeH}, we can fix a convergence radius $\delta>0$ for   $QH_{\rm big}(H_{n, m})$. 
We first assume the next proposition, whose proof  is given in   Appendix B. 
\begin{prop} \label{eigenv}
 Fix   $0<\delta' < \min \{\delta,1\}$.
 There exists   $0<\delta_2 < \frac{\delta'}{3}$ such that  if $|q_2| \in (0,\delta_2)$ and $|q_1|\in [\frac{\delta'}{3},\frac{\delta'}{2}], $ 
 then all the
 critical values $e_l$ of $f$, $0 \leq l \leq N-1$,   satisfy
\[
\big|e_{l}-n q_1^{\frac{1}{n}} \mu_l^{m+1} - (m+1) ~\mu_l q_1^{\frac{1}{N}} q_2^{\frac{1}{m+1}}\big|< C  \big|q_2^{\frac{2}{m+1}}\big|  
\] 
where $\mu_l=e^{\frac{2\pi \mathbf{i} l}{N}}$  and  $C$ is a constant depending only   on $n,m, \delta'$.
\end{prop}

Recall $\{\xi_j\}_{j=0}^{N-1}$ is a fixed homogeneous basis of $H^*(H_{n, m})$ with  $\xi_1=h_1$ and $\xi_2=h_2$. 
Denote $$D_b:=\bigg\{ \tau=\sum_{i=0}^{N-1}t_i \xi_i  \bigg| |t_0|,  |t_{3}|,...,|t_{N-1}|<\delta,  |q_1|<\frac{\delta'}{2},|q_2|<\delta_2 \bigg\}\subset H^{*}(H_{n,m}). $$

\begin{thm}\label{thm:GammaIImainbody}
    Gamma conjecture II holds for  $H_{n,m}$.
\end{thm}

\begin{proof}
     Take $(q_1',q_2')\in \mathbb{R}_{>0}^2$ with   $q_1'\in [\frac{\delta'}{3},\frac{\delta'}{2})$ and   $0<q_2'<<\delta_2$ (i.e. sufficiently small). Let $\tau_0=\log q_1'\cdot h_1+\log q_2' \cdot h_2\in V_{\mathbb{R}}=H^2(H_{n,m},\mathbb R)$.  By  \cref{eigenv},  the eigenvalues $\{e_{k}\}_{k=0}^{N-1}$  of $\hat{c}_1(\tau_0)$ are all pairwise distinct, which can be depicted as in Figure \ref{figure_eigen}. 
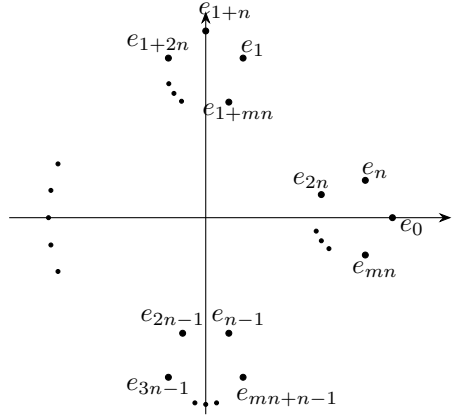
\begin{figure}[htbp]  
\centering
\begin{tikzpicture}[
    scale=1.3,
    >=Stealth,
    dot/.style={circle, fill, inner sep=1pt},
    cross/.style={cross out, draw, line width=0.8pt, minimum size=4pt, inner sep=0pt, outer sep=0pt},
    square/.style={rectangle, draw, fill=none, minimum size=3pt, inner sep=0pt},
    triangle/.style={regular polygon, regular polygon sides=3, draw, fill=none, minimum size=4pt, inner sep=0pt},
    diamond/.style={diamond, draw, fill=none, minimum size=4pt, inner sep=0pt},
    pentagon/.style={regular polygon, regular polygon sides=5, draw, fill=none, minimum size=4pt, inner sep=0pt},
    hexagon/.style={regular polygon, regular polygon sides=6, draw, fill=none, minimum size=4pt, inner sep=0pt}
]

\draw[->] (-2.0,0) -- (2.5,0) node[right] {};
\draw[->] (0,-2.0) -- (0,2.1) node[above] {};

    \def\r{0.4}

\filldraw (1.9, 0) circle [radius =0.03];
\draw (2.1,-0.1) node {\small $e_{0}$};
\filldraw (1.624, 0.380) circle [radius =0.03];
\draw (1.724, 0.480) node {\small $e_{n}$}; 
\filldraw (1.176, 0.235) circle [radius =0.03];
\draw (1.076, 0.385) node {\small $e_{2n}$}; 

 \foreach \a in {200, 216, 232} {
        \filldraw ({1.5 + \r*cos(\a)}, {0 + \r*sin(\a)}) circle [radius=0.02];
    }
\filldraw (1.624, -0.380) circle [radius =0.03];
\draw (1.724, -0.530) node {\small $e_{mn}$}; 

\filldraw (0,1.9) circle [radius =0.03];
\draw (0.2,2.1) node {\small $e_{1+n}$};
\filldraw ( 0.380,1.624) circle [radius =0.03];
\draw ( 0.480, 1.724) node {\small $e_{1}$}; 
\filldraw ( 0.235, 1.176) circle [radius =0.03];
\draw ( 0.335, 1.076) node {\small $e_{1+mn}$}; 

\foreach \a in {200, 216, 232} {
        \filldraw ({0 + \r*cos(\a)}, {1.5 + \r*sin(\a)}) circle [radius=0.02];
    }

\filldraw ( -0.380,1.624) circle [radius =0.03];
\draw ( -0.480, 1.774) node {\small $e_{1+2n}$}; 

\def\ra{1.6}

\foreach \a in {160,170, 180,190, 200} {
        \filldraw ({0 + \ra*cos(\a)}, {0 + \ra*sin(\a)}) circle [radius=0.02];
    }

\filldraw ( 0.380,-1.624) circle [radius =0.03];
\draw ( 0.830, -1.824) node {\small $e_{mn+n-1}$}; 
\filldraw ( 0.235, -1.176) circle [radius =0.03];
\draw ( 0.335, -1.026) node {\small $e_{n-1}$}; 
\filldraw ( -0.235, -1.176) circle [radius =0.03];
\draw ( -0.335,-1.026) node {\small $e_{2n-1}$}; 
\filldraw ( -0.380,-1.624) circle [radius =0.03];
\draw ( -0.480, -1.724) node {\small $e_{3n-1}$}; 

\foreach \a in {254, 270, 286} {
        \filldraw ({0 + \r*cos(\a)}, {-1.5 + \r*sin(\a)}) circle [radius=0.02];
    }

\end{tikzpicture}
\caption{Eigenvalues $e_k$ of $\hat{c}_1(\tau_0)$.}
\label{figure_eigen}
\end{figure}

Consider the paths $\rho_{ab}(t):[0,1]\to H^2(H_{n,m})$ defined by
    \begin{equation*}
     \rho_{ab}(t):=\tau_0+a \cdot 2\pi\mathbf{i}t \cdot h_1+b \cdot 2\pi\mathbf{i}t \cdot h_2,
    \end{equation*}
for $ 0\leq b \leq m, 0\leq a \leq n-1$.      The resulting $q_1(t)=q_1'e^{2 \pi\mathbf{i}  at}$ and $q_2(t)=q_2'e^{2\pi  \mathbf{i} b t}$ always have modulus $q_1'$ and $q_2'$, respectively.   Again by \cref{eigenv}, these paths are contained in  $$K_{ts}:=\{\eta \in H^{2}(H_{n,m}) | \mbox{ eigenvalues of} ~\hat c_1(\eta)  \mbox{are pairwise distinct } \}. $$ 
By  \cref{Gamma-I},   $H_{n,m}$ satisfies  {property Gamma-I} in the whole $V_{\mathbb{R}}$. Thus 
   $ \mathcal{Z}^{K,{\rm sm}}(\mathcal{O}_{H_{n,m}})(\tau_0,z)$ satisfies the following asymptotic behavior 
\begin{align*} 
         \mathcal{Z}^{K,{\rm sm}}(\mathcal{O}_{H_{n,m}})(\tau_0,z) e^{\frac{e_0}{z}} \rightarrow  \Psi_{N}(\tau_0), ~\text{as }  z \rightarrow 0 ~\text{ with } |\arg z -\theta_0| < \frac{\pi}{2}+\epsilon.
\end{align*}
Here $e_{0}$ is  the simple rightmost eigenvalue of $\hat c_1(\tau_0)$, and $\theta_{0}$ is the admissible phase sufficiently  close to $0$. 
By the arguments in \cite[Proposition 3.4.8]{GGI}, one can check that
 $(\Psi_{N}(\tau_0), \Psi_{N}(\tau_0))=1$  and there exists a $\widehat{\nabla}$-flat section $\widehat{y}_N(\tau_0,\lambda)$ which is holomorphic near $\lambda=e_0$  such that $\hat{y}_N(\tau_0,e_{0})=\Psi_{N}(\tau_0) $ and 
\begin{align} \label{GammaI}
   \mathcal{Z}^{K,{\rm sm}}(\mathcal{O}_{H_{n,m}})(\tau_0,z)=\frac{1}{z} \int_{e_{0}+\mathbb{R}_{\geq 0}e^{\mathbf{i} \theta_{0}}} \hat{y}_N(\tau_0,\lambda) e^{-\frac{\lambda}{z} }d\lambda.
\end{align}
   Note that $\tau_0 \in K_{ts}$
   and we can extend $\{\Psi_N\}$ to a normalized idempotent basis $\{\Psi_{j}\}_{j}$ of $(H^{*}(H_{n,m}),*_{\tau_0})$. 
\begin{figure}[htbp]
\centering
\begin{tikzpicture}[x=.90cm,y=.90cm,every node/.style={font=\fontsize{8}{9.5}\selectfont}]
\draw[milnor path,MilnorTeal] (1.900,0.000)
  -- (4.000,0.000);
\fill[MilnorTeal] (1.900,0.000) circle (1.1pt);
\node[anchor=west,inner sep=3pt] at (4.000,0.000) {$\mathcal O$};
\draw[milnor path,MilnorTeal] (1.624,0.380)
  .. controls (1.850,0.560) and (2.250,0.580) .. (2.800,0.580)
  -- (4.000,0.580);
\fill[MilnorTeal] (1.624,0.380) circle (1.1pt);
\node[anchor=west,inner sep=3pt] at (4.000,0.580) {$\mathcal O(0,-1)$};
\draw[milnor path,MilnorTeal] (1.176,0.235)
  .. controls (1.100,0.870) and (1.650,1.060) .. (2.800,1.060)
  -- (4.000,1.060);
\fill[MilnorTeal] (1.176,0.235) circle (1.1pt);
\node[anchor=west,inner sep=3pt] at (4.000,1.060) {$\mathcal O(0,-2)$};
\draw[milnor path,MilnorTeal] (1.624,-0.380)
  .. controls (1.150,-0.660) and (0.730,-0.370) .. (0.750,0.050)
  .. controls (0.780,0.940) and (1.500,1.550) .. (2.800,1.550)
  -- (4.000,1.550);
\fill[MilnorTeal] (1.624,-0.380) circle (1.1pt);
\node[anchor=west,inner sep=3pt] at (4.000,1.550) {$\mathcal O(0,-m)$};
\draw[milnor path,MilnorBlue] (0.380,1.624)
  .. controls (0.800,1.900) and (1.700,2.000) .. (2.800,2.000)
  -- (4.000,2.000);
\fill[MilnorBlue] (0.380,1.624) circle (1.1pt);
\node[anchor=west,inner sep=3pt] at (4.000,2.000) {$\mathcal O(-1,0)$};
\draw[milnor path,MilnorBlue] (0.000,1.900)
  .. controls (0.400,2.280) and (1.600,2.460) .. (2.800,2.460)
  -- (4.000,2.460);
\fill[MilnorBlue] (0.000,1.900) circle (1.1pt);
\node[anchor=west,inner sep=3pt] at (4.000,2.460) {$\mathcal O(-1,-1)$};
\draw[milnor path,MilnorBlue] (-0.380,1.624)
  .. controls (-0.530,2.330) and (0.800,2.950) .. (2.800,2.950)
  -- (4.000,2.950);
\fill[MilnorBlue] (-0.380,1.624) circle (1.1pt);
\node[anchor=west,inner sep=3pt] at (4.000,2.950) {$\mathcal O(-1,-2)$};
\draw[milnor path,MilnorBlue] (0.235,1.176)
  .. controls (-0.100,0.800) and (-0.770,1.030) .. (-0.750,1.600)
  .. controls (-0.750,2.750) and (0.300,3.450) .. (2.800,3.450)
  -- (4.000,3.450);
\fill[MilnorBlue] (0.235,1.176) circle (1.1pt);
\node[anchor=west,inner sep=3pt] at (4.000,3.450) {$\mathcal O(-1,-m)$};
\draw[milnor path,MilnorPlum] (0.235,-1.176)
  .. controls (-0.200,-0.820) and (-1.950,-1.180) .. (-1.950,-0.400)
  .. controls (-1.950,0.160) and (-1.950,0.800) .. (-1.950,1.600)
  .. controls (-1.950,3.300) and (0.050,3.980) .. (2.800,3.980)
  -- (4.000,3.980);
\fill[MilnorPlum] (0.235,-1.176) circle (1.1pt);
\node[anchor=west,inner sep=3pt] at (4.000,3.980) {$\mathcal O(1-n,0)$};
\draw[milnor path,MilnorPlum] (-0.235,-1.176)
  .. controls (-0.700,-1.370) and (-2.320,-1.180) .. (-2.320,-0.400)
  .. controls (-2.320,0.160) and (-2.320,0.800) .. (-2.320,1.600)
  .. controls (-2.320,3.660) and (0.050,4.450) .. (2.800,4.450)
  -- (4.000,4.450);
\fill[MilnorPlum] (-0.235,-1.176) circle (1.1pt);
\node[anchor=west,inner sep=3pt] at (4.000,4.450) {$\mathcal O(1-n,-1)$};
\draw[milnor path,MilnorPlum] (-0.380,-1.624)
  .. controls (-0.900,-1.900) and (-2.690,-1.180) .. (-2.690,-0.400)
  .. controls (-2.690,0.160) and (-2.690,0.800) .. (-2.690,1.600)
  .. controls (-2.690,4.020) and (0.050,4.940) .. (2.800,4.940)
  -- (4.000,4.940);
\fill[MilnorPlum] (-0.380,-1.624) circle (1.1pt);
\node[anchor=west,inner sep=3pt] at (4.000,4.940) {$\mathcal O(1-n,-2)$};
\draw[milnor path,MilnorPlum] (0.380,-1.624)
  .. controls (0.560,-2.050) and (-0.060,-2.250) .. (-0.550,-2.140)
  .. controls (-2.400,-1.990) and (-3.060,-1.080) .. (-3.060,-0.400)
  .. controls (-3.060,0.160) and (-3.060,0.800) .. (-3.060,1.600)
  .. controls (-3.060,4.350) and (0.050,5.440) .. (2.800,5.440)
  -- (4.000,5.440);
\fill[MilnorPlum] (0.380,-1.624) circle (1.1pt);
\node[anchor=west,inner sep=3pt] at (4.000,5.440) {$\mathcal O(1-n,-m)$};
\fill[MilnorTeal] ({1.5+.4*cos(200)},{0+.4*sin(200)}) circle (.65pt);
\fill[MilnorTeal] ({1.5+.4*cos(216)},{0+.4*sin(216)}) circle (.65pt);
\fill[MilnorTeal] ({1.5+.4*cos(232)},{0+.4*sin(232)}) circle (.65pt);
\fill[MilnorBlue] ({0+.4*cos(200)},{1.5+.4*sin(200)}) circle (.65pt);
\fill[MilnorBlue] ({0+.4*cos(216)},{1.5+.4*sin(216)}) circle (.65pt);
\fill[MilnorBlue] ({0+.4*cos(232)},{1.5+.4*sin(232)}) circle (.65pt);
\fill[MilnorPlum] ({0+.4*cos(254)},{-1.5+.4*sin(254)}) circle (.65pt);
\fill[MilnorPlum] ({0+.4*cos(270)},{-1.5+.4*sin(270)}) circle (.65pt);
\fill[MilnorPlum] ({0+.4*cos(286)},{-1.5+.4*sin(286)}) circle (.65pt);
\fill[black!80] ({1.6*cos(160)},{1.6*sin(160)}) circle (.65pt);
\fill[black!80] ({1.6*cos(170)},{1.6*sin(170)}) circle (.65pt);
\fill[black!80] ({1.6*cos(180)},{1.6*sin(180)}) circle (.65pt);
\fill[black!80] ({1.6*cos(190)},{1.6*sin(190)}) circle (.65pt);
\fill[black!80] ({1.6*cos(200)},{1.6*sin(200)}) circle (.65pt);
\fill[black!80] (3.400,1.230) circle (.4pt);
\fill[black!80] (3.400,1.305) circle (.4pt);
\fill[black!80] (3.400,1.380) circle (.4pt);
\fill[black!80] (3.400,3.125) circle (.4pt);
\fill[black!80] (3.400,3.200) circle (.4pt);
\fill[black!80] (3.400,3.275) circle (.4pt);
\fill[black!80] (3.400,3.640) circle (.4pt);
\fill[black!80] (3.400,3.715) circle (.4pt);
\fill[black!80] (3.400,3.790) circle (.4pt);

\fill[black!80] (3.400,5.115) circle (.4pt);
\fill[black!80] (3.400,5.190) circle (.4pt);
\fill[black!80] (3.400,5.265) circle (.4pt);
\node[font=\fontsize{7}{8}\selectfont,text=MilnorTeal] at (1.70,-.78) {$a=0$};
\node[font=\fontsize{7}{8}\selectfont,text=MilnorBlue] at (.30,.76) {$a=1$};
\node[font=\fontsize{7}{8}\selectfont,text=MilnorPlum] at (1.35,-1.75) {$a=n-1$};
\end{tikzpicture}
\caption{ Admissible paths $L_{a, b}$ with horizontal tails and the corresponding exceptional objects $\mathcal O(-a,-b):=p_1^*\mathcal O_{\mathbb P^n}(-a)\otimes p_2^*\mathcal O_{\mathbb P^m}(-b)$.  }
\label{figure_move}
\end{figure}
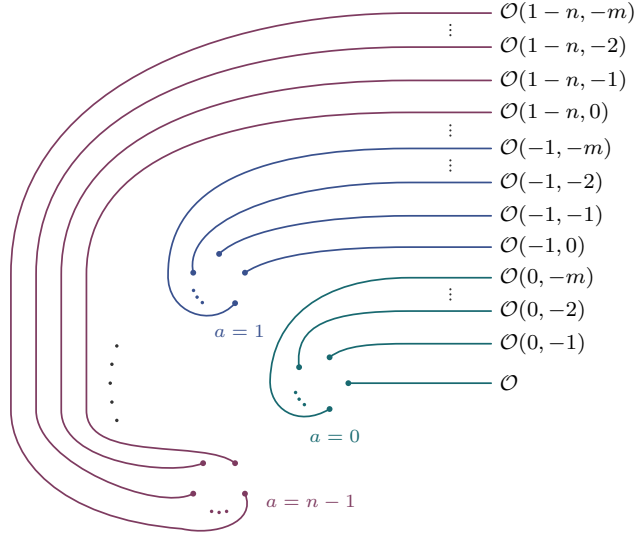

Denote $L_{0,0}:=e_{0}+\mathbb{R}_{\ge0} \rm e^{\mathbf{i} \theta_{0}}$. Consider the movement of eigenvalues along the paths $\rho_{ab}(t)$, and denote by $u^{l}_{ab}(t)$ the eigenvalue branch of $\hat{c}_1(\rho_{ab}(t))$ satisfying  $u^{l}_{ab}(0)=e_l$ for $0 \leq l \leq N-1$.  
By  \cref{eigenv}, we see that the displacement $$u^{l}_{ab}(t)-n{q'_1}^{\frac{1}{n}}e^{\frac{2\pi\mathbf{i}}{n}at}\mu_{l}^{m+1}$$ approximately traces a counterclockwise circular arc centered at the origin with radius $(m+1) {q'_1}^{\frac{1}{(m+1)n}}{q'_{2}}^{\frac{1}{m+1}}$,  and the moving eigenvalues remain pairwise distinct.  

As $\rho_{ab}(t)$ moves, we correspondingly rotate and deform $L_{0,0}$ to another admissible path such that the contour does not pass through any other eigenvalues   $u^{l}_{ab}(t)$ with $l \neq 0$ for $0\leq t <1$. And at $t=1$,   $L_{0,0}$ is eventually transformed to an admissible path  $L_{a,b}$  starting from $u^{0}_{ab}(1)=e_{bn+a}$.  Moreover, all admissible paths $\{ L_{a,b}\}_{a,b}$ are pairwise disjoint and have a common end-phase, and with respect to this common end-phase the order of their straight tails  aligns with the order of the full exceptional collection in \cref{prop:FEC} 
(see   Figure \ref{figure_move} for an illustration). Thus  $\mathcal{Z}^{K,{\rm sm}}(\mathcal{O}_{H_{n,m}})$ can be represented by $L_{a,b}$ at $\rho_{ab}(1)=\tau_0+a \cdot 2\pi\mathbf{i} \cdot h_1+b \cdot 2\pi\mathbf{i} \cdot h_2$ by analytic continuation from $\tau_0$ to $\rho_{ab}(1)$, and by the Galois action (\ref{Galois_ac}), we see that $\mathcal{Z}^{K,{\rm sm}}(p_1^{*}\mathcal{O}_{\mathbb{P}^{n}}(-a) \otimes p_2^{*}\mathcal{O}_{\mathbb{P}^{m}}(-b)   )$ can be represented by $L_{a,b}$ at $\tau_0$ for all $ 0\leq b \leq m, 0\leq a \leq n-1$.
Therefore for each $a,b$, there exists $j=(m+1)(n-a)-b$ such that 
$$\mathcal{Z}^{K,{\rm sm}}(p_1^{*}\mathcal{O}_{\mathbb{P}^{n}}(-a) \otimes p_2^{*}\mathcal{O}_{\mathbb{P}^{m}}(-b) )(\tau_0,z)=\frac{1}{z} \int_{L_{a,b}} \hat{y}_j(\tau_0,\lambda) e^{-\frac{\lambda}{z} }d\lambda.$$
Here $\hat{y}_j(\tau_0,\lambda) $ is a  $\widehat{\nabla}$-flat section  which is holomorphic near $\lambda=e_{bn+a}$ and $\hat{y}_j(\tau_0,e_{bn+a})=\Psi_{j}(\tau_0) $.
At last, 
recall that $B_{ts}$ is given in \cref{B_ts} and take a sufficiently small simply connected neighborhood  $ U'\subset D_b \cap B_{ts}$  of $\tau_0$.  Then the statement follows from     \cref{criterion} by taking the same phase as in Figure \ref{figure_move}.
\end{proof}

\section*{Appendix A: Ring presentation of $QH^*(H_{n,m})$}

In this appendix, we prove Proposition \ref{QH_Mil} by investigating quantum differential equations.

By the quantum Lefschetz principle (e.g. see \cite{CoGi}) for hypersurfaces of $\bbP^n\times \bbP^m$,   
   Givental's (small) $J$-function of $H_{n,m}$  is given by 
 \begin{align*}
     J^H:=J^{H_{n,m}}(\eta, z)= e^{\frac{\eta}{z}}  \sum_{d_1\geq 0, d_2\geq 0}^{\infty} \frac{q_1^{d_1}q_2^{d_2} \prod_{l=1}^{d_1+d_2}(h_1+h_2+lz) }{\prod_{l=1}^{d_1}(h_1+lz)^{n+1} \prod_{l=1}^{d_2}(h_2+lz)^{m+1}} .
 \end{align*}

\begin{prop} \label{QDE}
 We have 
$D_1(J^{H})=0$  and $D_2(J^{H})=0$, where   
\begin{align*}
     D_1&:=(z \frac{\partial}{\partial t_2})^{m+1}-e^{t_2}(z \frac{\partial}{\partial t_1}+z \frac{\partial}{\partial t_2})-ze^{t_2},\qquad \\
     D_2&:=\sum_{k=0}^{m} (-1)^{k} (z \frac{\partial}{\partial t_1})^{n-k}(z \frac{\partial}{\partial t_2})^{k}-e^{t_1}-(-1)^{m} e^{t_2} (z \frac{\partial}{\partial t_1})^{n-m}.
\end{align*}
\end{prop}

\begin{proof}
We prove the results by direct computation. 
\begin{align*}
    &e^{-\frac{t_1 h_1+t_2h_2}{z}}(z \frac{\partial}{\partial t_2})^{m+1}(J^{H}) \\
    =&
     \sum_{d_1 \geq 0,d_2 \geq 0}^{\infty} e^{d_1t_1+d_2t_2} \frac{ (h_2+d_2z)^{m+1}   \prod_{l=1}^{d_1+d_2}(h_1+h_2+lz) }{   \prod_{l=1}^{d_1}(h_1+lz)^{n+1} \prod_{l=1}^{d_2}(h_2+lz)^{m+1}} \\
     =&\sum_{d_1 \geq 0,d_2 \geq 1}^{\infty} e^{d_1t_1+d_2t_2} \frac{
    ((h_1+d_1z)+(h_2+d_2z))\prod_{l=1}^{d_1+d_2-1}(h_1+h_2+lz) }{   \prod_{l=1}^{d_1}(h_1+lz)^{n+1} \prod_{l=1}^{d_2-1}(h_2+lz)^{m+1}} \\
    =&e^{t_2}\sum_{d_1 \geq 0,d_2 \geq 1}^{\infty} e^{d_1t_1+(d_2-1)t_2} \frac{
    ((h_1+d_1z)+(h_2+(d_2-1)z+z))\prod_{l=1}^{d_1+d_2-1}(h_1+h_2+lz) }{   \prod_{l=1}^{d_1}(h_1+lz)^{n+1} \prod_{l=1}^{d_2-1}(h_2+lz)^{m+1}}\\
    =& e^{-\frac{t_1 h_1+t_2h_2}{z}} \Big(e^{t_2}(z \frac{\partial}{\partial t_1}+z \frac{\partial}{\partial t_2})(J^{H})+z e^{t_2} (J^{H})\Big).
\end{align*}
Here the second equality follows from $h_2^{m+1}=0$ in $H^{*}(H_{n,m})$ and  the cancellation of the term  $(h_2+d_2z)^{m+1}$.
Hence, $D_1(J^H)=0$.

Let $P=\sum_{k=0}^{m} (-1)^{k} (z \frac{\partial}{\partial t_1})^{n-k}(z \frac{\partial}{\partial t_2})^{k}$ be the quantization of the polynomial $P(x_1,x_2)=\sum_{k=0}^{m} (-1)^{k} (x_1)^{n-k}(x_2)^{k}$, and we denote  $W(x_1,x_2)=P(x_1,x_2)(x_1+x_2)=(x_1)^{n+1}+(-1)^{m}(x_1)^{n-m}(x_2)^{m+1}.$
    { 
    \begin{align*}
     &e^{-\frac{t_1 h_1+t_2h_2}{z}}P(J^{H}) \\
     =& \sum_{(d_1,d_2) \neq 0}^{\infty} e^{d_1t_1+d_2t_2}\frac{ P(h_1+d_1 z,h_2+d_2z)   \prod_{l=1}^{d_1+d_2}(h_1+h_2+lz) }{   \prod_{l=1}^{d_1}(h_1+lz)^{n+1} \prod_{l=1}^{d_2}(h_2+lz)^{m+1}} 
     +P(h_1,h_2)  \\
     =&\sum_{(d_1,d_2) \neq 0}^{\infty} e^{d_1t_1+d_2t_2}\frac{  W(h_1+d_1 z,h_2+d_2z)  \prod_{l=1}^{d_1+d_2-1}(h_1+h_2+lz) }{   \prod_{l=1}^{d_1}(h_1+lz)^{n+1} \prod_{l=1}^{d_2}(h_2+lz)^{m+1}} +0 \\
     =&e^{t_2}\sum_{ d_1=0,d_2 \geq 1}^{\infty} e^{(d_2-1)t_2}\frac{(-1)^{m} h_1^{n-m}  \prod_{l=1}^{d_2-1}(h_1+h_2+lz) }{    \prod_{l=1}^{d_2-1}(h_2+lz)^{m+1}} \\
 &+e^{t_1}\sum_{d_1 \geq 1,d_2=0}^{\infty} e^{(d_1-1)t_1}\frac{  \prod_{l=1}^{d_1-1}(h_1+h_2+lz) }{   \prod_{l=1}^{d_1-1}(h_1+lz)^{n+1}}  \\
     &+ e^{t_2} \sum_{d_1 \geq 1,d_2 \geq 1}^{\infty} e^{d_1t_1+(d_2-1)t_2}\frac{ (-1)^{m} (h_1+d_1z)^{n-m} \prod_{l=1}^{d_1+d_2-1}(h_1+h_2+lz) }{   \prod_{l=1}^{d_1}(h_1+lz)^{n+1} \prod_{l=1}^{d_2-1}(h_2+lz)^{m+1}} \\
     &+ e^{t_1} \sum_{d_1 \geq 1,d_2 \geq 1}^{\infty} e^{(d_1-1)t_1+d_2t_2}\frac{ \prod_{l=1}^{d_1+d_2-1}(h_1+h_2+lz) }{   \prod_{l=1}^{d_1-1}(h_1+lz)^{n+1} \prod_{l=1}^{d_2}(h_2+lz)^{m+1}}\\
    =&e^{-\frac{t_1 h_1+t_2h_2}{z}}\big(e^{t_2} (-1)^{m} (z \frac{\partial}{\partial t_1})^{n-m}(J^{H})+e^{t_1}(J^{H})\big). 
    \end{align*}
    } 
Here the  term $0$ in the second equality  follows from the relation $P(h_1,h_2)=0$ in $H^{*}(H_{n, m})$. 
 The third  equality holds by canceling out the same term $(h_1+d_1z)^{n+1}$ (resp. $(h_2+d_2z)^{m+1}$).  Combining the first term with the third term (resp. the second with the fourth) of the third equality yields the last equality. Hence, we have $D_2(J^H)=0$ by comparing $D_2$ with $P$.
\end{proof}

 \begin{proof}[Proof of Proposition \ref{QH_Mil}]
Derived from the structure of $H_{n, m}$ as a $\mathbb{P}^{n-1}$-bundle over $\mathbb{P}^m$, the classical cohomology  $H^*(H_{n, m})$ is canonically isomorphic to $\bbC[h_1, h_2]/I$ with $I=(f_1^q(h_1, h_2, 0,0), f_2^q(h_1, h_2, 0, 0))$ (which can also be read off from \cite[Proposition 5.4]{KLS}).
By   Proposition \ref{QDE} and   \cite[Corollary 6.4]{Giv},   the  relations $f_1^{q}(h_1,h_2,q_1,q_2)=0$ and $f_2^{q}(h_1,h_2,q_1,q_2)=0$ hold in $QH^*(H_{n, m})$. Hence, by   \cite[Proposition 2.2]{SiTi} for Fano manifolds, the quantum cohomology  $QH^*(H_{n, m})$ is isomorphic to the quotient ring
$\bbC[h_1, h_2, q_1, q_2]/(f_1^q(h_1, h_2, q_1,q_2), f_2^q(h_1, h_2, q_1, q_2)).$
 \end{proof}

\section*{Appendix B: Proof of \cref{eigenv}}
 
Recall $E(y)=n\frac{y^{m+1}}{q_2}+(m-n)y ,   F(y)= q_2 y^{n-m} ( \frac{y^{m}}{q_2}-1)^{n+1}-q_1$,  $N=n(m+1)$, $n_0=\frac{(n+1)(mn+2m-2n-1)}{2N}$.  Set $\kappa:=q_1^{1\over N}$ and  $\sigma:= q_2^{\frac{1}{m+1}}$  on a local simply connected domain.  Denote
\begin{align} \label{z_0}
z_{0,\kappa}(q_1,q_2):=\sum_{l=1}^{2}a_l(\kappa) \sigma^{l}, ~\mbox{with}~ a_1(\kappa)=\kappa, a_2(\kappa)=\frac{n+1}{N}\kappa^{1-m}.
\end{align}
 Recall that  $\delta>0$ is a fixed convergence radius of $QH_{\rm big}(H_{n, m})$.

\begin{lemma} \label{not root}
  For   $0<\delta' < \min \{\delta,1\}$, we denote $C_0:=2 \frac{n_0}{N}(\frac{3}{\delta'})^{\frac{2m-1}{N}}$.
  There exists   $0<\delta_0 <\frac{\delta'}{3}$ such that  if $|q_2| \in (0,\delta_0)$, $|q_1|\in [\frac{\delta'}{3},\frac{\delta'}{2}]$ and  $\theta \in \mathbb{R}$ all hold, then  the complex number $z_{0,\kappa}+ C_0 e^{\mathbf{i} \theta}\sigma^{3}$   is not a root of $F(y)$.  
\end{lemma}

\begin{proof} Denote by $Q(\kappa, \sigma, e^{\mathbf{i}\theta})$ the polynomial in $\mathbb{C}[\kappa, \kappa^{-1}, \sigma, e^{\mathbf{i}\theta}]$ such that 
   \begin{align*}
       F(z_{0, \kappa}+ C_0e^{\mathbf{i} \theta}\sigma^{3})&=\sigma^{2}\left((N\kappa^{2m-1} C_0e^{\mathbf{i} \theta}+ n_0)\kappa^{N-2m} +\sigma Q(\kappa, \sigma, e^{\mathbf{i}\theta})\right).
   \end{align*}

Assume $F(z_{0, \kappa}+ C_0e^{\mathbf{i} \theta}\sigma^{3})=0$. Then it follows from   $q_2 \neq 0$ that
$$(N\kappa^{2m-1} C_0e^{\mathbf{i} \theta}+ n_0)\kappa^{N-2m} +\sigma Q(\kappa, \sigma, e^{\mathbf{i}\theta})=0.   $$
Since $ (\frac{\delta'}{3})^{\frac{1}{N}} \leq |\kappa|\leq  (\frac{\delta'}{2})^{\frac{1}{N}}$ by the hypothesis and $|e^{\mathbf{i}\theta}|=1$, it follows that the polynomial $ Q(\kappa, \sigma, e^{\mathbf{i}\theta})$ is  uniformly bounded, independent of $q_1$. Hence, 
 there exists  $\delta_0<\frac{\delta'}{3}$ such that for any  $|q_2| \in (0,\delta_0)$, we have 
$$|\sigma Q(\kappa, \sigma, e^{\mathbf{i}\theta})|<\frac{1}{2} (\frac{\delta'}{3})^{\frac{N-2m}{N}}n_0 .$$ 
This 
results in  a contradiction using
\begin{align*}
    \big|\big(N \kappa^{2m-1}C_0 e^{\mathbf{i} \theta}+n_0 \big)\kappa^{N-2m}\big|\geq & |\kappa|^{N-2m} (N|\kappa|^{2m-1} C_0-n_0)\geq (\frac{\delta'}{3})^{\frac{N-2m}{N}}n_0.
\end{align*}  
\end{proof}

\begin{lemma} \label{two_term}
With the notation as well as the range of $(q_1, q_2)$ in 
  \cref{not root}, we 
fix branches of $\kappa=q_1^{1/N}$ and $\sigma=q_2^{1/(m+1)}$.
 Let
$x_\kappa(q_1,q_2)$ be the continuous root branch of $F(x)$
near $z_{0,\kappa}(q_1,q_2)$. Then
\[
|x_\kappa(q_1,q_2)-z_{0,\kappa}(q_1,q_2)|
<
C_0|\sigma^3|.
\]
\end{lemma}

\begin{proof} 

Assume that there exists $(\bar q_1,\bar q_2)$ in the same range as $(q_1, q_2)$ (i.e. $|\bar q_2|\in(0,\delta_0)$ and
$|\bar q_1|\in[\frac{\delta'}{3},\frac{\delta'}{2}]$) 
such that
\[
|x_\kappa(\bar q_1,\bar q_2)-z_{0,\kappa}(\bar q_1,\bar q_2)|
\ge C_0|\bar q_2|^{\frac{3}{m+1}}.
\]
For fixed $\bar q_1$, by \cref{asym of cri_point}, there exists
$\delta_2\in(0,|\bar q_2|)$ such that for every $q'_2$ with
$0<|q'_2|<\delta_2$,
\[
|x_\kappa(\bar q_1,q'_2)-z_{0,\kappa}(\bar q_1,q'_2)|
<
C_0|q'_2|^{\frac{3}{m+1}}.
\]
Take one such  $q'_2$ with  $\arg q_2'=\arg \bar q_2$.
For $r\in[|q'_2|,|\bar q_2|]$, we define  
\[
\tilde g(r)
=
|x_\kappa(\bar q_1,q_2(r))-z_{0,\kappa}(\bar q_1,q_2(r))|
-
C_0|q_2(r)|^{\frac{3}{m+1}},
\]
where 
 $
q_2(r):=r e^{\mathbf i\arg \bar q_2}$. 
Since the root branch $x_\kappa$ is continuous in $q_2$, the function
$\tilde g$ is continuous on $[|q'_2|,|\bar q_2|]$. Moreover,
$ 
\tilde g(|q'_2|)<0, \tilde g(|\bar q_2|)\ge 0.
$ 
Hence there is $\tilde r_0\in[|q'_2|,|\bar q_2|]$ such that $\tilde g(\tilde r_0)=0$.
Let $q_{20}=q_2(\tilde r_0)$. Then
\[
|x_\kappa(\bar q_1,q_{20})-z_{0,\kappa}(\bar q_1,q_{20})|
=
C_0|q_{20}|^{\frac{3}{m+1}},
\]
which
contradicts Lemma \ref{not root}.
\end{proof}

\subsubsection*{Proof of \cref{eigenv}  }

 Note that the critical values are given by 
 $e_l=E(y_{l,\kappa})$
 with  $y_{l,\kappa}$ being the roots of  $F(y)$. The analytic expansion of $y_{l,\kappa}$ with respect to   $\sigma$ is  of the form $$y_{l,\kappa}=\sum\nolimits_{i \geq 1}a_{i}(\kappa) \sigma^{i}  =\sigma\left(P(\kappa,\sigma)+R(\kappa,\sigma)\right), $$
 where $P(\kappa,\sigma)=a_1(\kappa)+a_2(\kappa)\sigma$ and
$R(\kappa,\sigma):=\sum_{i \geq 3}a_{i}\sigma^{i-1}$. Indeed,  
\begin{align*}
    &E(y_{l,\kappa}) \\
    =& n\Big( P(\kappa,\sigma)^{m+1}+\sum_{j=1}^{m+1} \binom{m+1}{j}P^{m+1-j} R^{j} \Big) +(m-n)\sigma \left(P(\kappa,\sigma)+R(\kappa,\sigma)\right) \\
    =& n\kappa^{m+1}+(m+1)\kappa \cdot \sigma 
+\sigma^2Q_3(\kappa,\sigma)+ Q_2(\kappa,\sigma,R(\kappa,\sigma) ),
\end{align*}
where $ Q_2(\kappa,\sigma,R(\kappa,\sigma)):=n\cdot \sum_{j=1}^{m+1}\binom{m+1}{j}P^{m+1-j} R^{j}  +(m-n)\sigma \cdot R$ and $$Q_3(\kappa,\sigma)=n\sum_{i=2}^{m+1} \binom{m+1}{i} a_{1}(\kappa)^{m+1-i} a_2(\kappa)^{i} \sigma^{i-2}+(m-n)a_2(\kappa) .$$
Recall $z_{0,\kappa}(q_1,q_2)=\sigma\cdot P(\kappa,\sigma)$ is given in  \cref{z_0}.
By   \cref{two_term}, $|y_{l,\kappa}-z_{0,\kappa}(q_1,q_2)|=|\sigma R(\kappa,\sigma)| < C_0 |\sigma^3|$ for any  $q_1,q_2$ satisfying $|q_2| \in (0,\delta_0)$ and $|q_1|\in [\frac{\delta'}{3},\frac{\delta'}{2}]$.  Since $ (\frac{\delta'}{3})^{\frac{1}{N}} \leq |\kappa|\leq  (\frac{\delta'}{2})^{\frac{1}{N}} $,   the coefficients  of $\sigma^{i}$ in $Q_2(\kappa,\sigma,R(\kappa,\sigma))$ and $Q_3(\kappa,\sigma)$  are uniformly bounded.
Thus there exist  $\delta_2\in (0, \delta_0)$ and  constants $C_1,C_2>0$ (depending only on $\delta',\delta_0,m,n$) such that   $$| Q_2(\kappa,\sigma,R(\kappa,\sigma))|< C_1 |R(\kappa,\sigma)|<C_0 C_1 |\sigma|^2 \qquad \mbox{and} \quad |Q_3(\kappa,\sigma)| < C_2. $$
We conclude
  \begin{align*}
      &|E(y_{l,\kappa})-n\kappa^{m+1}-(m+1)\kappa \cdot \sigma|< (C_0C_1+C_2) |\sigma|^2.
  \end{align*}
 Take $\kappa=q_1^{\frac{1}{(m+1)n}}\mu_l$, $C=C_0C_1+C_2$. Then  the statement follows.


\begin{thebibliography}{99}


\bibitem[AF03]{AF} M. J. Ablowitz and A. S. Fokas. Complex Variables: Introduction and Applications,
2nd ed., Cambridge Texts Appl. Math., Cambridge Univ. Press, Cambridge,
2003. 



 

\bibitem[BCFKS00]{BCFKS} V. Batyrev, I. Ciocan-Fontanine, B. Kim and  D. van Straten. \,{\it Mirror symmetry and toric degenerations of partial flag manifolds}. Acta Math. 184, no. 1, 1--39, 2000.

  


 
\bibitem[CK99]{CoKa}  D. A. Cox and  S. Katz.
{\it Mirror symmetry and algebraic geometry}. 
Mathematical Surveys and Monographs, 68. American Mathematical Society, Providence, RI, 1999. 



 

\bibitem[CDG24]{CDG}  G. Cotti,  B.  Dubrovin  and D.  Guzzetti.
{\it Helix structures in quantum cohomology of Fano varieties}.
Lecture Notes in Mathematics, 2356. Springer, 2024. 



 
 

 

\bibitem[Chow25]{Chow}C. H. Chow. {\it Gamma conjecture I for flag varieties}, preprint at arXiv: math.AG/2501.13221. 


\bibitem[CG07]{CoGi}  T. Coates and A. Givental.
{\it Quantum Riemann-Roch, Lefschetz and Serre}. 
Ann. of Math. (2) 165 (2007), no. 1, 15–53.


\bibitem[Du98]{Dub}  B. Dubrovin.
{\it Geometry and analytic theory of Frobenius manifolds.} 
Proceedings of the International Congress of Mathematicians, Vol. II (Berlin, 1998).
Doc. Math. 1998, Extra Vol. II, 315--326.



\bibitem[Dub99]{Dub99} B. Dubrovin. \emph{Painlev\'e transcendents in two-dimensional topological field theory}, In: The Painlev\'e
Property, CRM Ser. Math. Phys., Springer, (New York 1999), pp. 287--412.

 

\bibitem[FZ19]{FaZh} B. Fang and  P. Zhou.  {\it Gamma II for toric varieties from integrals on T-dual branes and homological mirror symmetry}, preprint at arXiv:1903.05300.

 

\bibitem[GGI16]{GGI}  S. Galkin,  V.  Golyshev and  H. Iritani.
{\it Gamma classes and quantum cohomology of Fano manifolds: gamma conjectures}.
Duke Math. J. 165 (2016), no. 11, 2005--2077.

\bibitem[GI19]{GaIr}  S. Galkin and H. Iritani. {\it Gamma conjecture via mirror symmetry}. Primitive forms and related subjects—Kavli IPMU 2014, 55--115, Adv. Stud. Pure Math., 83, Math. Soc. Japan, Tokyo, 2019.

\bibitem[GHIKLS24]{GHIKLS} S. Galkin, J. Hu, H. Iritani, H.-Z. Ke, C. Li and Z. Su. {\it Revisiting Gamma conjecture I: counterexamples and modifications}, to appear in Compositio Math.; preprint at arXiv: 2405.16979.

 


\bibitem[Gi96]{Giv}  A. Givental. {\em
Equivariant Gromov-Witten invariants}.
Internat. Math. Res. Notices 1996, no. 13, 613--663.

\bibitem[Giv97]{Giv97} A. Givental. 
{\it Stationary phase integrals, quantum Toda lattices, flag manifolds and the mirror conjecture}. Topics in singularity theory, 103–115,
Amer. Math. Soc. Transl. Ser. 2, 180, Adv. Math. Sci., 34, Amer. Math. Soc., Providence, RI, 1997.

\bibitem[Giv98]{Giv98}  A.  Givental.
{\it A mirror theorem for toric complete intersections}.  Topological field theory, primitive forms and related topics (Kyoto, 1996), 141–175,
Progr. Math., 160, Birkhäuser Boston, Boston, MA, 1998.


 \bibitem[GZ16]{GZ16}V. Golyshev and D. Zagier. \emph{\it Proof of the gamma conjecture for Fano 3-folds with a Picard lattice of rank one},   Izv. Math. 80 (2016), no. 1, 24--49.


 
 

\bibitem[GS25]{GS25} A. Gyenge and S. Szab\'o. {\it Blow-ups and the quantum spectrum of surfaces}, Adv. Math. 479(2025), Part A, Paper No. 110432, 39pp.




\bibitem[HKLS25]{HKLS}  J. Hu, H.-Z. Ke, C. Li and L. Song. {\em Mirror symmetry for certain blowups of Grassmannians}, to appear in Trans. Amer. Math. Soc. ; arXiv:  {math.AG/2502.13545}.

\bibitem[HKLSu26]{HKLS2} J. Hu,  H.-Z. Ke, C. Li and Z. Su. {\it Gamma conjecture II via global Gamma-I}, preprint at arXiv: math.AG/2606.07418.
 
 

\bibitem[HK23]{HuKe} X. Hu and H.-Z. Ke.
{\it Gamma conjecture II for quadrics}. 
Math. Ann. 387 (2023), no. 1-2, 927--983.

\bibitem[HKLY21]{HKLY}
J. Hu, H.-Z. Ke, C. Li and T. Yang.
\emph{Gamma conjecture I for del Pezzo surfaces}.
Adv. Math. 386 (2021), 107797.


\bibitem[HS07]{HeSe}
 C. Hertling and  C. Sevenheck.
{\it Nilpotent orbits of a generalization of Hodge structures}.
J. Reine Angew. Math. 609 (2007), 23–80.

\bibitem[Iri07]{Iri07} H. Iritani. \emph{Convergence of quantum cohomology by quantum Lefschetz}, J. Reine Angew. Math. 610 (2007), 29--69.

\bibitem[Ir09]{Iri}  H. Iritani.
{\it An integral structure in quantum cohomology and mirror symmetry for toric orbifolds}. 
Adv. Math. 222 (2009), no. 3, 1016--1079.

\bibitem[Ir11]{Iri11} H. Iritani.
{\it Quantum cohomology and periods}. 
Ann. Inst. Fourier (Grenoble) 61 (2011), no. 7, 2909–2958.


\bibitem[Ir20]{Iri20} H. Iritani.
{\it Global mirrors and discrepant transformations for toric Deligne-Mumford stacks}. 
SIGMA Symmetry Integrability Geom. Methods Appl. 16 (2020), Paper No. 032, 111 pp.

\bibitem[IK23]{IrKo} H. Iritani and Y. Koto. {\it  Quantum cohomology of projective bundles}, to appear in Geom. Topol.; preprint at arXiv:math.AG/2307.03696v4. 
 

\bibitem[Iri23]{Iri23}  H. Iritani. {\it Quantum cohomology of blowups}, to appear in Kyoto J. Math.; preprint at arXiv: math.AG/2307.13555.

\bibitem[Iri23b]{Iri23icm}
H.~Iritani.
\emph{Gamma classes and quantum cohomology},
ICM -- International Congress of Mathematicians. Vol. IV. Sections 5-8, 2552--2574.
EMS Press, Berlin, 2023.

\bibitem[Ke24]{Ke24}H.-Z. Ke.\,{\it Conjecture $\mathcal{O}$ for projective complete intersections}, Int. Math. Res. Not. IMRN 2024, no. 5, 3947-3974.


\bibitem[KLS25]{KLS} H.-Z. Ke, C. Li and J. Song. {\it Towards small quantum Chern character}, preprint at arXiv: math.AG/2507.12333. 

\bibitem[Kont94]{Kont}  M.  Kontsevich. {\it Homological algebra of mirror symmetry}. Proceedings of the International Congress of Mathematicians, Vol. 1, 2 (Zürich, 1994), 120–139, Birkhäuser, Basel, 1995. 

\bibitem[KKP08]{KKP}  L. Katzarkov,  M. Kontsevich and  T. Pantev.
{\it Hodge theoretic aspects of mirror symmetry}.  From Hodge theory to integrability and TQFT tt*-geometry, 87–174,
Proc. Sympos. Pure Math., 78, Amer. Math. Soc., Providence, RI, 2008.

\bibitem[LRY26]{LRY} C. Li,  K. Rietsch and M. Yang. {\em A Peterson program for general    Schubert varieties and mirror symmetry},  preprint at arXiv: math.AG/2609.30242.

 

 


\bibitem[Or92]{Orl}  D. O. Orlov.
{\it Projective bundles, monoidal transformations, and derived categories of coherent sheaves}. (Russian summary)
Izv. Ross. Akad. Nauk Ser. Mat. 56 (1992), no. 4, 852--862; translation in
Russian Acad. Sci. Izv. Math. 41 (1993), no. 1, 133--141.

 

\bibitem[SS20]{SS20} F. Sanda and Y. Shamoto. \emph{An analogue of Dubrovin’s conjecture}, Ann. Inst. Fourier, 70 (2020) no.2, 621–682.
 


\bibitem[ST97]{SiTi}  B. Siebert and  G.  Tian.
{\it On quantum cohomology rings of Fano manifolds and a formula of Vafa and Intriligator}.
Asian J. Math. 1 (1997), no. 4, 679--695.

\bibitem[Yan22]{Yan22} Z. Yang.\,{\it Gamma conjecture I for blowing up $\mathbb{P}^n$ along $\mathbb{P}^r$},  preprint at arXiv: math.AG/2202.04234.



\end{thebibliography}
\end{document}